\documentclass[12pt,reqno]{amsart}
\usepackage{a4wide}
\usepackage{microtype}
\usepackage[all]{xy}
\usepackage{amsmath}
\usepackage{amsfonts}
\usepackage{amssymb}
\usepackage{amsthm}
\usepackage{amscd}
\usepackage{here}
\usepackage[
bookmarks=false,
colorlinks=true,
debug=true,
naturalnames=true,
pdfnewwindow=true,
citecolor=blue,
linkcolor=blue,
urlcolor = blue]{hyperref}
\usepackage[alphabetic,msc-links,initials]{amsrefs}
\usepackage{verbatim}
\usepackage{graphicx,xcolor}
\usepackage{mathrsfs}
\usepackage{upgreek}
\usepackage{tikz}
\usepackage[colorinlistoftodos]{todonotes}
\usepackage{mathtools}
\newcommand{\seq}{\coloneqq}
\usepackage{multirow}
\usepackage{yfonts}

\numberwithin{equation}{section}

\newtheorem{Thm}{Theorem}[section]
\newtheorem{Cor}[Thm]{Corollary}
\newtheorem{Prop}[Thm]{Proposition}
\newtheorem{Lem}[Thm]{Lemma}

\theoremstyle{definition}
\newtheorem{Def}[Thm]{Definition}
\newtheorem{Rem}[Thm]{Remark}
\newtheorem{Ex}[Thm]{Example}
\newtheorem*{DefThm}{Definition \& Claim}

\newcommand{\Hom}{\mathop{\mathrm{Hom}}\nolimits}

\newcommand{\Aut}{\mathop{\mathrm{Aut}}\nolimits}

\newcommand{\Ext}{\mathop{\mathrm{Ext}}\nolimits}
\newcommand{\Ker}{\mathop{\mathrm{Ker}}\nolimits}

\newcommand{\gext}{\mathop{\mathrm{ext}}\nolimits}
\newcommand{\gtor}{\mathop{\mathrm{tor}}\nolimits}
\newcommand{\diag}{\mathop{\mathrm{diag}}\nolimits}

\newcommand{\rank}{\operatorname{\mathrm{rank}}}

\newcommand{\lcm}{\mathop{\mathrm{lcm}}\nolimits}

\newcommand{\N}{\mathbb{N}}
\newcommand{\Z}{\mathbb{Z}}
\newcommand{\Q}{\mathbb{Q}}

\newcommand{\F}{\mathbb{F}}
\newcommand{\kk}{\Bbbk}
\newcommand{\bD}{\mathbb{D}}
\newcommand{\se}{\mathrm{s}}
\newcommand{\tl}{\mathrm{t}}
\newcommand{\op}{\mathrm{op}}

\newcommand{\id}{\mathrm{id}}

\newcommand{\fg}{\mathfrak{g}}

\newcommand{\fd}{\textfrak{d}}
\newcommand{\sQ}{\mathsf{Q}}

\newcommand{\tC}{\widetilde{C}}

\newcommand{\tQ}{\widetilde{Q}}
\newcommand{\tPi}{\widetilde{\Pi}}
\newcommand{\tG}{\widetilde{G}}
\newcommand{\ul}{\underline}

\newcommand{\ep}{\varepsilon}

\newcommand{\umod}{\text{-$\mathrm{mod}$}}
\newcommand{\mmu}{\ul{\mu}}

\newcommand{\hK}{\hat{K}}
\newcommand{\Lotimes}{\overset{\mathbf{L}}{\otimes}}
\newcommand{\CKP}{C^{\mathrm{KP}}}


\thanks{R.F. was supported by JSPS Overseas Research Fellowships, and JSPS KAKENHI Grant Number JP23K12955 during the revision.}
\thanks{K.M. was supported by the Kyoto Top Global University project, Grant-in-Aid for JSPS Fellows (JSPS KAKENHI Grant Number JP21J14653, JP23KJ0337 during the revision) and JSPS bilateral program (Grant
Number JPJSBP120213210).}

\title[DCM \and GPA]
{Deformed Cartan matrices and generalized preprojective algebras II: General type}

\date{\today}

\author[R.~Fujita]{Ryo Fujita}
\address[R.~Fujita]{Research Institute for Mathematical Sciences, Kyoto University, Oiwake-Kitashirakawa, Sakyo, Kyoto, 606-8502, Japan \& Institut de Math\'{e}matiques de Jussieu-Paris Rive Gauche, Universit\'{e} Paris Cit\'{e}, F-75013, Paris, France}
\email{rfujita@kurims.kyoto-u.ac.jp}

\author[K.~Murakami]{Kota Murakami}
\address[K.~Murakami]{Graduate School of Mathematical Sciences, The University of Tokyo
3-8-1 Komaba, Meguro, Tokyo, 153-8914 Japan \& Department of Mathematics, Kyoto University, Kitashirakawa Oiwake-cho, Sakyo-ku, Kyoto, 606-8502, Japan}
\email{murakami@ms.u-tokyo.ac.jp}

\begin{document} 
\maketitle
\begin{abstract}
We propose a definition of deformed symmetrizable generalized Cartan matrices with several deformation parameters, which admit a categorical interpretation by graded modules over the generalized preprojective algebras in the sense of Gei\ss-Leclerc-Schr\"oer.
Using the categorical interpretation, we deduce a combinatorial formula for the inverses of our deformed Cartan matrices in terms of braid group actions. 
Under a certain condition, which is satisfied in all the symmetric cases or in all the finite and affine cases, our definition coincides with that of the mass-deformed Cartan matrices introduced by Kimura-Pestun in their study of quiver $\mathcal{W}$-algebras. 
\end{abstract}
\tableofcontents

\section*{Introduction}  

In their study of the deformed $\mathcal{W}$-algebras, Frenkel-Reshetikhin~\cite{FR98} introduced a certain $2$-parameter deformation $C(q,t)$ of the Cartan matrix of finite type.
In the previous work~\cite{FM}, the present authors gave a categorical interpretation of this deformed Cartan matrix $C(q,t)$ in terms of bigraded modules over the generalized preprojective algebras in the sense of Gei\ss-Leclerc-Schr\"oer~\cite{GLS}.
More precisely, we have shown that the entries of the matrix $C(q,t)$ and its inverse $\tC(q,t)$ can be expressed by the Euler-Poincar\'e pairings of certain bigraded modules.

The definition of the generalized preprojective algebra is given in a generality of arbitrary symmetrizable Kac-Moody type by \cite{GLS}, and it admits a Weyl group symmetry~\cite{GLS,AHIKM} and a geometric realization of crystal bases \cite{GLS4}. As a sequel of \cite{FM}, the main purpose of the present paper is to propose a categorification of a several parameter deformation of arbitrary symmetrizable generalized Cartan matrix (GCM for short) by considering multi-graded modules over the generalized preprojective algebra.
In the context of theoretical physics, Kimura-Pestun~\cite{KP1, KP2} introduced \emph{the mass-deformed Cartan matrix}, a deformation of GCM with several deformation parameters, in their study of (fractional) quiver $\mathcal{W}$-algebras, which is a generalization of Frenkel-Reshetikhin's deformed $\mathcal{W}$-algebras.
Our deformation essentially coincides with Kimura-Pestun's mass-deformed Cartan matrix under a certain condition which is satisfied in all the symmetric cases or in all the finite and affine cases (see \S\ref{Ssec:KP}).

To explain our results more precisely, let us prepare some kinds of terminology.
Let $C = (c_{ij})_{i,j \in I}$ be a GCM with a symmetrizer $D = \diag(d_i \mid i \in I)$.
We put $g_{ij} \seq \gcd(|c_{ij}|, |c_{ji}|)$ and $f_{ij} \seq |c_{ij}|/g_{ij}$ for $i,j \in I$ with $c_{ij} < 0$.
Associated with these data, we have the generalized preprojective algebra $\Pi$ defined over an arbitrary field (see \cite{GLS} for the precise definition or \S\ref{Ssec:GPA} for our convention).
We introduce the (multiplicative) abelian group $\Gamma$ generated by the elements 
\[ \{q,t\} \cup \{ \mu_{ij}^{(g)} \mid (i,j, g) \in I\times I \times \Z, c_{ij} < 0, 1 \le g \le g_{ij} \} \]  
which subject to the relations 
\[ \mu^{(g)}_{ij} \mu^{(g)}_{ji} = 1 \quad \text{for all $i, j \in I$ with $c_{ij} < 0$ and $1 \le g \le g_{ij}$}.\]
These elements play the role of deformation parameters.
Here, we introduced the parameters $\mu_{ij}^{(g)}$ in addition to $q$ and $t$ inspired by \cite{KP1, KP2}, where the counterparts are called \emph{mass-parameters}.
We endow a certain $\Gamma$-grading on the algebra $\Pi$ as in \eqref{eq:deg} below.
We can show that this grading on $\Pi$ is universal under a reasonable condition, see \S\ref{subsec:univ_grad}. With the terminology, we give the following definition of $(q, t, \mmu)$-deformation $C(q,t, \mmu)$ of GCM $C$, and propose a categorical framework which organizes some relevant combinatorics in terms of the $\Gamma$-graded $\Pi$-modules:
\begin{DefThm}
We define the $\Z[\Gamma]$-valued $I \times I$-matrix $C(q,t, \mmu)$ by the formula
\begin{equation} \label{eq:defC}
C_{ij}(q,t, \mmu) = \begin{cases} 
q^{d_i}t^{-1} + q^{-d_i}t & \text{if $i=j$};
\\ - [f_{ij}]_{q^{d_i}} \sum_{g = 1}^{g_{ij}}\mu_{ij}^{(g)} & \text{if $c_{ij} < 0$}; 
\\ 0 & \text{otherwise},
\end{cases}
\end{equation}
where $[k]_q = (q^{k}-q^{-k})/(q-q^{-1})$ is the standard $q$-integer.
We establish the following statements:
\begin{enumerate}
\item Each entry of $C(q,t, \mmu)$ and its inverse $\tC(q,t,\mmu)$ can be expressed as the Euler-Poincar\'e paring of certain $\Gamma$-graded $\Pi$-modules. (\S\ref{Ssec:EP})
\item Moreover, when $C$ is of infinite type, the formal expansion at $t=0$ of each entry of $\tC(q,t,\mmu)$ coincides with the $\Gamma$-graded dimension of a certain $\Pi$-module, and hence its coefficients are non-negative. (Corollary~\ref{Cor:chi})
\item\label{main:3} For general $C$, the formal expansion at $t=0$ of $\tC(q,t,\mmu)$ admits a combinatorial expression in terms of a braid group symmetry. 
(\S\S\ref{Ssec:cinv} \&  \ref{sec:braid})
\end{enumerate} 
\end{DefThm}

Note that if we consider the above~\eqref{main:3} for each finite type and some specific reduced words, then it recovers the combinatorial formula obtained by \cite{HL15} and \cite{KO} after some specialization. We might see our generalization as a kind of aspects of the Weyl/braid group symmetry of $\Pi$ about general reduced expressions (e.g. \cite{FG,M}).
When $C$ is of finite type, these results are essentially same as the results in our previous work \cite{FM}.

When $C$ is of infinite type, the algebra $\Pi$ is no longer finite-dimensional.
In this case, we find it suitable to work with the category of $\Gamma$-graded modules which are bounded from below with respect to the $t$-grading, and its \emph{completed} Grothendieck group.   
Then, the discussion is almost parallel to the case of finite type. 
Indeed, we give a uniform treatment which deals with the cases of finite type and of infinite type at the same time.      

In the case of finite type, the above combinatorial aspects of the deformed Cartan matrices play an important role in the representation theory of quantum loop algebras, see our previous work~\cite{FM} and references therein. 
We may expect that our results here on the deformed GCM are also useful in the study of quiver $\mathcal{W}$-algebras and the representation theory of quantum affinizations of Kac-Moody algebras in the future.

This paper is organized as follows.
In \S\ref{Sec:Cartan}, after fixing our notation, we discuss combinatorial aspects (i.e., a braid group action in \S\ref{Ssec:Braid} and the formula for $\tC(q,t,\mmu)$ using it in \S\ref{Ssec:cinv}) of our deformed Cartan matrices.  
The proofs of several assertions require the categorical interpretation and hence are postponed to the next section.
In \S\ref{Sec:gpa}, we discuss the categorical interpretation of our deformed GCM in terms of the graded modules over the generalized preprojective algebras.
The final \S\ref{Sec:Rem} consists of three remarks, which are logically independent from the other parts of the paper.
In \S\ref{Ssec:KP}, we compare our deformed GCM with the mass-deformed Cartan matrix in the sense of Kimura-Pestun~\cite{KP2}.
In \S\ref{subsec:univ_grad}, we show that our $\Gamma$-grading on $\Pi$ is universal among the gradings valued at free abelian groups.
In \S\ref{Ssec:species}, we briefly discuss the $t$-deformed GCM, which is obtained from our $C(q,t,\mmu)$ by evaluating all the deformation parameters except for $t$ at $1$, and its categorical interpretation by the classical representation theory of modulated graphs in the sense of Dlab-Ringel~\cite{DR}.
\subsection*{Conventions}
Throughout this paper, we use the following conventions. 
\begin{itemize}
\item For a statement $\mathrm{P}$, we set $\delta(\mathrm{P})$ to be $1$ or $0$ according that $\mathrm{P}$ is true or false. 
We often use the abbreviation $\delta_{x,y} \seq \delta(x=y)$ known as Kronecker's delta.
\item For a group $G$, let $\Z[G]$ denote the group ring and $\Z[\![G]\!]$ the set of formal sums $\{\sum_{g \in G} a_g g \mid a_g \in \Z\}$. 
Note that $\Z[\![ G ]\!]$ is a $\Z[G]$-module in the natural way.  
If $\Z[G]$ is a commutative integral domain, we write $\Q(G)$ for its fraction field. 
\end{itemize}

\section{Deformed Cartan matrices}
\label{Sec:Cartan}

\subsection{Notation}
\label{Ssec:notation}

Let $I$ be a finite set.
Recall that a $\Z$-valued $I \times I$-matrix $C = (c_{ij})_{i,j \in I}$ is called \emph{a symmetrizable generalized Cartan matrix}
if the following conditions are satisfied:
\begin{enumerate}
\renewcommand{\theenumi}{\rm C\arabic{enumi}}
\item \label{C1} $c_{ii} = 2$, $c_{ij} \in \Z_{\le 0}$ for all $i, j \in I$ with $i\neq j$, and $c_{ij} = 0$ if and only if $c_{ji} = 0$,
\item \label{C2} there is a diagonal matrix $D = \diag(d_i \mid i \in I)$ with $d_i \in \Z_{>0}$ for all $i \in I$ such that the product $DC$ is symmetric.
\end{enumerate}  
We call the diagonal matrix $D$ in \eqref{C2} \emph{a symmetrizer of $C$}.  
It is said to be \emph{minimal} when $\gcd(d_i \mid i \in I) =1$.
For $i, j \in I$, we write $i \sim j$ when $c_{ij} < 0$.   
We say that a symmetrizable generalized Cartan matrix $C$ is \emph{irreducible} if, for any $i, j \in I$, there is a sequence $i_1, \ldots, i_l \in I$ satisfying $i \sim i_1 \sim \cdots \sim i_l \sim j$.
In this case, a minimal symmetrizer of $C$ is unique, and any symmetrizer of $C$ is a scalar multiple of it.
From now on, by a GCM,
we always mean an irreducible symmetrizable generalized Cartan matrix.
We say that $C$ is of finite type if it is positive definite, and it is of infinite type otherwise.

Throughout this section, we fix a GCM $C=(c_{ij})_{i,j \in I}$ with its symmetrizer $D= \diag(d_i\mid i \in I)$.
For any $i, j \in I$ with $i \sim j$, we set 
\[ g_{ij} \seq \gcd(|c_{ij}|, |c_{ji}|), \quad f_{ij} \seq |c_{ij}|/g_{ij}, \quad d_{ij} \seq \gcd(d_i, d_j). \]  
By definition, we have $g_{ij} = g_{ji}, d_{ij} = d_{ji}$ and $f_{ij} = d_j/d_{ij}$.
Let $r \seq \lcm(d_i \mid i \in I)$.
We note that the transpose ${}^{\mathtt{t}}{C} = (c_{ji})_{i,j \in I}$ is also a GCM, whose minimal symmetrizer is $rD^{-1} = \diag(r/d_i \mid i\in I)$.  
Following \cite{GLS}, we say that a subset $\Omega \subset I \times I$ is \emph{an acyclic orientation of $C$} if the following conditions are satisfied: 
\begin{itemize}
\item $\{(i,j), (j,i)\} \cap \Omega \neq \varnothing$ if and only if $i \sim j$,
\item for any sequence $(i_1, i_2, \ldots, i_l)$ in $I$ with $l > 1$ and $(i_k, i_{k+1}) \in \Omega$ for all $1 \le k < l$, we have $i_1 \neq i_l$.
\end{itemize}

Let $\sQ = \bigoplus_{i \in I}\Z \alpha_i$ be the root lattice of the Kac-Moody algebra associated with $C$, where $\alpha_i$ is the $i$-th simple root for each $i \in I$.  
We write $s_i$ for the $i$-th simple reflection, which is an automorphism of $\sQ$ given by $s_i \alpha_j = \alpha_j - c_{ij}\alpha_i$ for $j \in I$.
The Weyl group $W$ is defined to be the subgroup of $\Aut(\sQ)$ generated by all the simple reflections $\{ s_i\}_{i \in I}$.
The pair $(W,\{s_i\}_{i \in I})$ forms a Coxeter system.     

\subsection{Deformed Cartan matrices}\label{Ssec:DCM}

Let $\Gamma$ be the (multiplicative) abelian group defined in Introduction.
As an abelian group, $\Gamma$ is free of finite rank. 
Let $\mmu^\Z$ denote the subgroup of $\Gamma$ generated by all the elements in $\{\mu_{ij}^{(g)} \mid i,j \in I, i \sim j, 1 \le g \le g_{ij}\}$.
Then we have $\Gamma = q^\Z \times t^\Z \times \mmu^\Z$, where $x^{\Z} \seq \{ x^k \mid k \in \Z\}$. 
If we choose an acyclic orientation $\Omega$ of $C$, we have
$\mmu^\Z = \prod_{(i,j)\in \Omega} \prod_{g=1}^{g_{ij}} (\mu_{ij}^{(g)})^\Z$.
In particular, the rank of $\Gamma$ is $2 + \sum_{(i,j) \in \Omega} g_{ij}$.
Consider the group ring $\Z[\Gamma]$ of $\Gamma$. 
Given an acyclic orientation $\Omega$ of $C$, it is identical to the ring of Laurent polynomials in the variables $q,t$ and $\mu_{ij}^{(g)}$ with $(i,j) \in \Omega$.  

We define \emph{the deformed generalized Cartan matrix} (\emph{deformed GCM} for short) $C(q,t, \mmu)$ to be the $\Z[\Gamma]$-valued $I \times I$-matrix whose $(i,j)$-entry $C_{ij}(q,t, \mmu)$ is given by
the formula~\eqref{eq:defC} in Introduction.
We often evaluate all the parameters $\mu_{ij}^{(g)}$ at $1$ and write $C(q,t)$ for the resulting $\Z[q^{\pm 1}, t^{\pm 1}]$-valued matrix.
More explicitly, its $(i,j)$-entry is given by
\[
C_{ij}(q,t) \seq \delta_{i,j} (q^{d_i}t^{-1} + q^{-d_i}t) - \delta(i \sim j)
g_{ij}[f_{ij}]_{q^{d_i}}. \]   
We refer to the matrix $C(q,t)$ as the \emph{$(q,t)$-deformed GCM}.
Note that we have $[d_i]_q C_{ij}(q,t) = g_{ij}[d_i f_{ij}]_q$ whenever $i\neq j$, and hence the matrix $([d_i]_q C_{ij}(q,t))_{i,j \in I}$ is symmetric.

\begin{Rem}
When the GCM $C$ is of finite type, the matrix $C(q,t)$ coincides with the $(q,t)$-deformed Cartan matrix considered in \cite{FR98}. 
A deformed GCM of general
type is also considered in \cite{KP1, KP2}, called \emph{the mass deformed Cartan matrix}. 
We discuss the difference between our definition and the definition in \cite{KP2} in \S\ref{Ssec:KP}. 
\end{Rem} 

Let $\Gamma_0 \seq q^\Z \times \mmu^\Z \subset \Gamma$.
Since $\Gamma = t^\Z \times \Gamma_0$, we have $\Z[\Gamma] = \Z[\Gamma_0][t^{\pm 1}]$. 
Letting $q^{\pm D} \seq \diag(q^{\pm d_i} \mid i \in I)$, we can write
\begin{equation} \label{eq:X}
C(q,t,\mmu) = (\id-tX)t^{-1}q^{D},
\end{equation}
for some $\Z[\Gamma_0][t]$-valued matrix $X$.
In particular, the matrix $C(q,t,\mmu)$ is invertible as a $\Z[\Gamma_0](\!(t)\!)$-valued matrix and its inverse $\tC(q,t,\mmu) = (\tC_{ij}(q,t,\mmu))$ is given by 
\[ \tC(q,t,\mmu) = q^{-D}t\left( \id + \sum_{k > 0} X^k t^{k} \right).\]
\begin{Ex}
    Even if we begin with a non-invertible GCM $C$, we obtain $C(q, t, \mmu)$ as an invertible matrix. For example, if we take
    \[C=\begin{pmatrix}
        2 & -2 \\
        -2 & 2
    \end{pmatrix}\quad \text{and}\quad D=
    \diag(1, 1),\]
    then we obtain
    \begin{equation} \label{eq:A11}
C(q,t,\mmu)=
\begin{pmatrix}
    qt^{-1}+q^{-1}t & -(\mu_{12}^{(1)}+\mu_{12}^{(2)})\\
    -(\mu_{21}^{(1)}+\mu_{21}^{(2)}) & qt^{-1}+q^{-1}t
\end{pmatrix}.
    \end{equation}
    Since $\det C(q,t,\mmu) = q^2 t^{-2}- (\mu_{12}^{(1)}\mu_{21}^{(2)}+\mu_{21}^{(1)}\mu_{12}^{(2)})+q^{-2}t^2 \in \Z[\Gamma_0](\!(t)\!)^{\times}$, our $C(q,t,\mmu)$ is invertible.
\end{Ex}

\begin{Thm} \label{Thm:tCpos}
When $C$ is of infinite type, the matrix $\tC(q,t,\mmu)$ has non-negative coefficients, namely we have $\tC_{ij}(q,t,\mmu) \in \Z_{\ge 0}[\Gamma_0][\![t]\!]$ for any $i,j \in I$.
\end{Thm}

A proof will be given in the next section (see Corollary~\ref{Cor:chi}~\eqref{Cor:chi:2} below).

\begin{Rem}
If we evaluate all the deformation parameters except for $q$ at $1$ in \eqref{eq:A11}, we get a $q$-deformed Cartan matrix $C(q)$, which is different from the naive $q$-deformation $C'(q)$, where
\[C(q) = \begin{pmatrix} [2]_q&-2 \\ -2 & [2]_q\end{pmatrix}, \qquad  
C'(q) = \begin{pmatrix} [2]_q& [-2]_q \\ [-2]_q & [2]_q \end{pmatrix}.\]
Note that $C(q)$ is invertible, while $C'(q)$ is not invertible. 
See also Remark~\ref{Rem:qC} below for a related discussion on $q$-deformed Cartan matrices.
In the context of the representation theory of quantum affinizations, the choice of $q$-deformation of GCM affects the definition of the algebra.   
For the quantum affinization of  $\widehat{\mathfrak{sl}}_2$, the matrix $C(q)$ was used by Nakajima \cite[Remark 3.13]{Nak11} and also adopted by Hernandez in \cite{Her11}.  
See \cite[Remark 4.1]{Her11}.
\end{Rem}

\subsection{Braid group actions}\label{Ssec:Braid}
Let $\Q(\Gamma)$ denote the fraction field of $\Z[\Gamma]$.
Let $\phi$ be the automorphism of the group $\Gamma$ given by $\phi(q) = q$, $\phi(t) = t$, and $\phi(\mu_{ij}^{(g)}) = \mu_{ji}^{(g)}$ for all possible $i,j \in I$ and $g$. 
It induces the automorphisms of $\Z[\![\Gamma]\!]$ and $\Q(\Gamma)$, for which we again write $\phi$. 
We often write $a^{\phi}$ instead of $\phi(a)$.

Consider the $\Q(\Gamma)$-vector space $\sQ_\Gamma$ given by 
\[ \sQ_\Gamma \seq \Q(\Gamma)\otimes_{\Z}\sQ = \bigoplus_{i \in I} \Q(\Gamma)\alpha_i.\]
We endow $\sQ_\Gamma$ with a non-degenerate $\phi$-sesquilinear hermitian form $(-,-)_\Gamma$ by
\[ (\alpha_i, \alpha_j)_\Gamma \seq [d_i]_q C_{ij}(q,t, \mmu)\]
for each $i,j \in I$.
Here the term ``$\phi$-sesquilinear hermitian" means that it satisfies
\[ (ax,by)_{\Gamma} = a^\phi b (x,y)_\Gamma, \quad (x,y)_\Gamma = (y,x)_\Gamma^\phi \]
for any $x,y \in \sQ_\Gamma$ and $a,b \in \Q(\Gamma)$.  
Let $\{\alpha_i^\vee\}_{i \in I}$ be another basis of $\sQ_\Gamma$ defined by 
\[
\alpha_i^\vee \seq q^{-d_i}t[d_i]_q^{-1} \alpha_i.
\] 
It is thought of a deformation of simple coroots.
We have 
\[(\alpha_i^\vee, \alpha_j)_\Gamma = q^{-d_i}tC_{ij}(q,t,\mmu)\] 
for any $i,j \in I$.
Let $\{ \varpi_i^\vee \}_{i\in I}$ denote the dual basis of $\{\alpha_i\}_{i \in I}$ with respect to $(-,-)_\Gamma$. 
We also consider the element $\varpi_i \seq [d_i]_q \varpi_i^\vee$ for each $i \in I$.
With these conventions, we have 
\begin{equation} \label{eq:ao}
\alpha_i = \sum_{j \in I}C_{ji}(q,t,\mmu)\varpi_j, \qquad
\alpha_i^\vee = q^{-d_i}t\sum_{j \in I} C_{ij}(q,t,\mmu)^\phi \varpi_j^\vee.
\end{equation}

For each $i \in I$, we define a $\Q(\Gamma)$-linear automorphism $T_i$ of $\sQ_\Gamma$ by
\begin{equation} \label{Baction}
T_i x \seq x - (\alpha_i^\vee, x)_\Gamma \alpha_i 
\end{equation}
for $x \in \sQ_\Gamma$.
In terms of the basis $\{ \alpha_i \}_{i \in I}$, we have
\begin{equation} \label{eq:Troot}
T_i \alpha_j = \alpha_j - q^{-d_i}tC_{ij}(q,t,\mmu) \alpha_i.
\end{equation}
Thus, the action \eqref{Baction} can be thought of a deformation of the $i$-th simple reflection $s_i$.
Note that our $\mathbb{Q}(\Gamma)$-linear automorphisms $T_i\,(i\in I)$ of $\sQ_\Gamma$ recover the braid group actions that were introduced in \cite{Cha} and \cite{BP} for finite type cases after certain specializations (see \cite[Section~1.3]{FM}).


\begin{Prop}
\label{Prop:braidrel}
The operators $\{ T_i \}_{i \in I}$ define an action of the braid group associated to the Coxeter system $(W, \{s_i\}_{i \in I})$, i.e., they satisfy the braid relations
:
\begin{alignat*}{2}
T_i T_j &= T_j T_i &\qquad & \text{if $c_{ij}=0$}, \\
T_i T_j T_i &= T_j T_i T_j && \text{if $c_{ij}c_{ji}=1$}, \\
(T_i T_j)^k &= (T_j T_i)^k && \text{if $c_{ij}c_{ji}=k$ with $k \in \{2,3\}$}. 
\end{alignat*}
\end{Prop}
A proof will be given in \S\ref{sec:braid} below (after Lemma~\ref{Lem:JT}). 

Given $w \in W$, we choose a reduced expression $w = s_{i_1}s_{i_2} \cdots s_{i_l}$ and set $T_w \seq T_{i_1} T_{i_2} \cdots T_{i_l}$.
By Proposition~\ref{Prop:braidrel}, $T_w$ does not depend on the choice of reduced expression.

\subsection{Remark on finite type}
In this subsection, we assume that $C$ is of finite type. 
Since we always have $g_{ij}=1$ in this case, we write $\mu_{ij}$ instead of $\mu_{ij}^{(1)}$.
For any $(i,j) \in I$, we define $\mu_{ij} \seq \mu_{i,i_1} \mu_{i_1,i_2} \cdots \mu_{i_k,j}$, where $(i_1, \ldots, i_k)$ is any finite sequence in $I$ such that $i \sim i_1 \sim i_2 \sim \cdots \sim i_k\sim j$.  
Note that the element $\mu_{ij} \in \Gamma$ does not depend on the choice of such a sequence.
Let $[-]_{\mmu =1} \colon \Z[\Gamma] \to \Z[q^{\pm 1}, t^{\pm 1}]$ denote the map induced from the specialization $\mmu^{\Z} \to \{1\}$. 
Recall $C_{ij}(q,t) = [C_{ij}(q,t,\mmu)]_{\mmu = 1}$ by definition.

\begin{Lem} 
\label{Lem:fin}
When $C$ is of finite type, for any $i,j \in I$ and a sequence $(i_1,\ldots,i_k)$, we have
\[ (\varpi_i^\vee, T_{i_1} \cdots T_{i_k} \alpha_{j})_\Gamma = \mu_{ij} [(\varpi_i^\vee, T_{i_1} \cdots T_{i_k} \alpha_{j})_\Gamma]_{\mmu=1}.\]
\end{Lem}
\begin{proof}
By definition, we have $C_{ij}(q,t,\mmu) = \mu_{ij}C_{ij}(q,t)$ for any $i,j \in I$.
Then the assertion follows from \eqref{eq:Troot}.
\end{proof}

Let $w_0 \in W$ be the longest element.
It induces an involution $i \mapsto i^*$ of $I$ by $w_0 \alpha_i = - \alpha_{i^*}$. 
We consider the $\Q(\Gamma)$-linear automorphism $\nu$ of $\sQ_\Gamma$ given by $\nu(\alpha_i) = \mu_{i^*i} \alpha_{i^*}$ for each $i \in I$.
It is easy to see that $\nu$ is involutive and the pairing $(-,-)_\Gamma$ is invariant under $\nu$. 
In particular, we have $\nu (\varpi_i^\vee) = \mu_{ii^*}\varpi_{i^*}^\vee$ for each $i \in I$.
Denote the Coxeter and dual Coxeter numbers associated with $C$ by $h$ and $h^\vee$ respectively.

\begin{Prop} \label{Prop:longest}
Assume that $C$ is of finite type. 
We have $T_{w_0} = - q^{-rh^\vee}t^h \nu.$
\end{Prop}
\begin{proof}
We know that the assertion holds when $\mmu = 1$ \cite[Theorem~1.6]{FM}. 
It lifts to the desired formula thanks to Lemma~\ref{Lem:fin}.
\end{proof}

\subsection{Combinatorial inversion formulas}
\label{Ssec:cinv}
Let $C$ be a GCM of general type.

Let $(i_k)_{k \in \Z_{>0}}$ and $(j_k)_{k \in \Z_{>0}}$ be two sequences in $I$.
We say that $(i_k)_{k \in \Z_{>0}}$ is 
\emph{commutation-equivalent} to $(j_k)_{k \in \Z_{>0}}$ if there is a bijection $\sigma \colon \Z_{> 0} \to \Z_{>0}$ such that $i_{\sigma(k)} = j_k$ for all $k \in \Z_{>0}$ and we have $c_{i_k, i_l} =0$ whenever $k < l$ and $\sigma(k)> \sigma(l)$.
\begin{Thm} \label{Thm:inv1}
Let $(i_k)_{k \in \Z_{>0}}$ be a sequence in $I$ satisfying the following condition:
\begin{enumerate}
\item if $C$ is of finite type, $(i_k)_{k \in \Z_{>0}}$ is commutation-equivalent to another sequence $(j_k)_{k \in \Z_{>0}}$ such that the subsequence $(j_1, \ldots, j_l)$ is a reduced word with $l$ being the length of the longest element $w_0 \in W$ and we have $j_{k+l} = j_{k}^*$ for all $k \in \Z_{>0}$;
\item if $C$ is of infinite type, the subsequence $(i_1, i_2, \ldots, i_k)$ is a reduced word for all $k \in \Z_{>0}$, and we have $|\{ k \in \Z_{>0} \mid i_k = i\}| = \infty$ for each $i \in I$.
\end{enumerate}
Then, for any $i, j \in I$, we have
\begin{equation} \label{eq:inv1}
\tC_{ij}(q,t, \mmu) = q^{-d_j}t\sum_{k \in \Z_{>0}; i_k = j}(\varpi_i^\vee, T_{i_1} \cdots T
_{i_{k-1}}\alpha_j)_\Gamma.
\end{equation}  
\end{Thm}
\begin{proof}[Proof of Theorem~\ref{Thm:inv1} for finite type]
Note that the RHS of \eqref{eq:inv1} is unchanged if we replace the sequence $(i_1,i_2,\ldots)$ with another  commutation-equivalent sequence thanks to Proposition~\ref{Prop:braidrel}.  
When $C$ is of finite type, we know that the equality \eqref{eq:inv1} holds at $\mmu=1$ by \cite[Proposition 3.16]{FM}.
Since we have $\tC_{ij}(q,t,\mmu) = \mu_{ij}\tC_{ij}(q,t)$ for any $i,j \in I$, we can deduce  \eqref{eq:inv1} for general $\mmu$ thanks to Lemma~\ref{Lem:fin}.
\end{proof}

A proof when $C$ is of infinite type will be given in \S\ref{sec:braid} below (after Corollary~\ref{Cor:braid:ac}).

In the remaining part of this section, we discuss the special case of the above inversion formula~\eqref{eq:inv1} when the sequence comes from a Coxeter element and deduce a recursive algorithm to compute $\tC(q,t,\mmu)$. 
Fix an acyclic orientation
$\Omega$ of $C$. 
We say that a total ordering $I =\{i_1, \ldots, i_n\}$ is \emph{compatible with $\Omega$} if $(i_k, i_l) \in \Omega$ implies $k < l$. 
Taking a compatible total ordering, we define the Coxeter element $\tau_\Omega \seq s_{i_1}\cdots s_{i_n}$.
The assignment $\Omega \mapsto \tau_\Omega$ gives a well-defined bijection between the set of acyclic orientations of $C$ and the set of Coxeter elements of $W$.
In what follows, we abbreviate $T_\Omega \seq T_{\tau_\Omega}$.
Letting $I = \{i_1,\ldots,i_n \}$ be a total ordering compatible with $\Omega$, for each $i \in I$, we set
\begin{equation} \label{eq:hgamma} 
\beta^\Omega_i \seq (1-T_\Omega)\varpi_i = q^{-d_i}tT_{i_1} T_{i_2} \cdots T_{i_{k-1}}\alpha_{i_k} \qquad \text{if $i = i_k$}.
\end{equation} 
Note that the resulting element $\beta^\Omega_i$ is independent of the choice of the compatible ordering.
\begin{Prop} \label{Prop:inv2} 
Let $\Omega$ be an acyclic orientation of $C$.
For any $i,j \in I$, we have
\begin{equation} \label{eq:inv2}
\tC_{ij}(q,t,\mmu) = \sum_{k=0}^\infty (\varpi_i^\vee, T^{k}_\Omega \beta^\Omega_j)_\Gamma.
\end{equation}
\end{Prop}
\begin{proof}
Choose a total ordering $I = \{i_1, \ldots, i_n\}$ compatible with $\Omega$. Then we have $T_\Omega = T_{i_1} \cdots T_{i_n}$. 
We extend the sequence $(i_1, \ldots, i_n)$ to an infinite sequence $(i_k)_{k \in \Z_{>0}}$ so that $i_{k+n} = i_k$ for all $k \in \N$.
When $C$ is of infinite type, this sequence satisfies the condition in Theorem~\ref{Thm:inv1} by~\cite{Speyer}, and hence we obtain \eqref{eq:inv2}. 
When $C$ is of finite type, we know that the subsequence $(i_1,\ldots, i_{2l}) = (i_1, \ldots, i_n)^{h}$ is commutation-equivalent to a sequence $(j_1, \ldots, j_{2l})$ such that $(j_1,\ldots,j_l)$ is a reduced word (adapted to $\Omega$) for the longest element $w_0$ and $j_{k+l} = j_k^*$ for all $1 \le k \le l$.  
Indeed, when $C$ is of simply-laced type, it follows from~\cite{Bed}. When $C$ is of non-simply-laced type, we simply have $\tau_\Omega^{h/2} = w_0$ and $(i_1, \ldots, i_n)^{h/2}$ is a reduced word for $w_0$. 
Therefore Theorem~\ref{Thm:inv1} again yields \eqref{eq:inv2}.
\end{proof}

\begin{Lem} \label{Lem:mesh1}
For each $i \in I$ and $k \in \N$, we have
\begin{equation} \label{eq:rec1}
q^{d_i} t^{-1} T^{k+1}_\Omega \beta^\Omega_i + q^{-d_i}t T^{k}_\Omega \beta^\Omega_i + \sum_{j \sim i}C_{ji}(q,t,\mmu) T_\Omega^{k+\delta((j,i)\in \Omega)}\beta^\Omega_j =0.
\end{equation} 
\end{Lem}
\begin{proof}
For any $i, j \in I$, we have 
\[ T_i \varpi_j = \begin{cases}
-q^{-2d_i}t^{2}\varpi_i - q^{-d_i}t\sum_{i' \sim i}C_{i'i}(q,t,\mmu)\varpi_{i'} & \text{if $i=j$}, \\
\varpi_j & \text{if $i\neq j$}
\end{cases}
\] 
by definition. Using this identity, we obtain
\[ q^{d_i}t T_\Omega \varpi_i =- q^{-d_i}t\varpi_i -\sum_{j \sim i} C_{ji}(q,t,\mmu) T_\Omega^{\delta((j,i)\in \Omega)}\varpi_j.\]
Applying $T^{k}_\Omega(1-T_\Omega)$ yields \eqref{eq:rec1}.
\end{proof}

Once we fix a total ordering $I= \{i_1,\ldots,i_n\}$ compatible with $\Omega$, the equalities \eqref{eq:hgamma} and \eqref{eq:rec1} compute the elements $T^{k}_\Omega \beta^\Omega_i$ for all $(k,i) \in \Z_{\ge 0} \times I$ recursively along the lexicographic total ordering of $\Z_{\ge 0} \times I$. 
Thus, together with \eqref{eq:inv2}, we have obtained a recursive algorithm to compute $\tC_{ij}(q,t, \mmu)$. 

We say that a GCM $C$ is \emph{bipartite} if there is a function $\epsilon \colon I \to \Z / 2\Z$ such that $\epsilon(i) = \epsilon(j)$ implies $i \not\sim j$. 
When $C$ is bipartite, we can simplify the above recursive formula by separating the parameter $t$ as explained below.

For each $i \in I$, we consider a $\Q(\Gamma)$-linear automorphism $\bar{T}_i$ of $\sQ_\Gamma$ obtained from $T_i$ by evaluating the parameter $t$ at $1$.
More precisely, it is given by
\[
\bar{T}_i \alpha_j = \alpha_j - q^{-d_i}C_{ij}(q,1,\mmu) \alpha_i
\]
for all $j \in I$.
The operators $\{\bar{T}_i\}_{i \in I}$ define another action of the braid group, under which the $\Q(\Gamma_0)$-subspace $\sQ_{\Gamma_0} \seq \bigoplus_{i \in I}\Q(\Gamma_0)\alpha_i$ of $\sQ_\Gamma$ is stable.
\begin{Def}
A function $\xi \colon I \to \Z$ is called \emph{a height function} (for $C$) if 
\[ |\xi(i) - \xi(j)| = 1 \quad \text{for all $i, j \in I$ with $i \sim j$}. \] 
A height function $\xi$ gives an acyclic orientation $\Omega_\xi$ of $C$ such that we have $(i, j) \in \Omega_\xi$ if $i \sim j$ and $\xi(j) = \xi(i) + 1$.
When $i \in I$ is a sink of $\Omega_\xi$, in other words, when $\xi(i) < \xi(j)$ holds for all $j \in I$ with $j \sim i$, we define another height function $s_i \xi$ by 
\[ (s_i \xi)(j) \seq \xi(j)+2\delta_{i,j}.\]
\end{Def}

\begin{Rem} There exists a height function for $C$ if and only if $C$ is bipartite.
\end{Rem}

Given a function $\xi \colon I \to \Z$, we define a linear automorphism $t^\xi$ of $\sQ_{\Gamma}$ by 
\[ t^\xi \alpha_i \seq t^{\xi(i)}\alpha_i \]
for each $i \in I$. 
When $\xi\colon I \to \Z$ is a height function and $i \in I$ a sink of $\Omega_\xi$, a straightforward computation yields
$t^{\xi} T_i = \bar{T}_i t^{s_i \xi}$, from which we deduce 
\begin{equation} \label{eq:hT&T} t^{\xi} T_{\Omega_\xi} = \bar{T}_{\Omega_\xi} t^{\xi + 2}.
\end{equation}

\begin{Def}
Let $\xi \colon I \to \Z$ be a height function.
Define a map $\Phi_{\xi} \colon I \times \Z \to \sQ_{\Gamma_0}$ by
\[ 
\Phi_\xi(i,u) \seq \begin{cases}
\bar{T}^{k}_{\Omega_\xi}(1-\bar{T}_{\Omega_\xi})\varpi_i & \text{if $u = \xi(i) + 2k$ for some $k \in \Z_{\ge 0}$,} \\
0 & \text{else}.
\end{cases}
\]
\end{Def}
The next proposition is a consequence of Proposition~\ref{Prop:inv2} and \eqref{eq:hT&T}.
\begin{Prop} \label{Prop:inv3}
Let $\xi \colon I \to \Z$ be a height function.
For any $i,j \in I$, we have
\begin{equation} \label{eq:inv3}
\tC_{ij}(q,t,\mmu) = \sum_{u = \xi(j)}^\infty\left(\varpi_i^\vee, \Phi_\xi(j,u)\right)_{\Gamma}t^{u-\xi(i)+1}. 
\end{equation}
\end{Prop}

Now, Lemma~\ref{Lem:mesh1} specializes to the following.

\begin{Lem} \label{Lem:mesh2}
Let $\xi \colon I \to \Z$ be a height function.
For any $(i,u) \in I \times \Z$ with $u > \xi(i)$, we have
\begin{equation} \label{eq:Phi}
q^{-d_i} \Phi_\xi(i,u-1)+q^{d_i}\Phi_\xi(i,u+1) + \sum_{j \sim i} C_{ji}(q,1,\mmu) \Phi_\xi(j,u) =0. 
\end{equation}
In particular, \eqref{eq:Phi} enables us to compute recursively all the $\Phi_\xi(i,u)$ starting from 
\[\Phi_{\xi}(i, \xi(i)) = (1-\bar{T}_{\Omega_\xi})\varpi_i= q^{-d_i}\bar{T}_{i_1}\cdots \bar{T}_{i_{k-1}}\alpha_i \quad \text{for all $i \in I$},\]
where $I = \{i_1,\ldots,i_n\}$ is a total ordering compatible with $\Omega_\xi$ and $i_k=i$. 
\end{Lem}

Thus, Proposition~\ref{Prop:inv3} combined with Lemma~\ref{Lem:mesh2} gives a simpler recursive algorithm to compute $\tC_{ij}(q,t,\mmu)$ when $C$ is bipartite. 

\begin{Rem}
When $C$ is of finite type, the formula~\eqref{eq:inv3} recovers the formulas in \cite[Proposition 2.1]{HL15} (type ADE) and \cite[Theorem 4.7]{KO} (type BCFG) after the specialization $\Gamma_0 \to \{1\}$.
\end{Rem}

\begin{Rem}
When $C$ is of finite type, the above algorithm can be used to compute $\tC(q,t,\mmu)$ (or $\tC(q,t)$) completely.
For example, let us give an explicit formula of $\tC(q,t)$ for type $\mathrm{F}_4$. 
We use the convention $I = \{1,2,3,4\}$ with $1 \sim 2 \sim 3 \sim 4$ and $(d_1,d_2,d_3,d_4) = (2,2,1,1)$.
Then, for any $i \le j$, we have \[\tC_{ij}(q,t) = \frac{f_{ij}(q,t)+f_{ij}(q^{-1},t^{-1})}{q^9t^{-6}+q^{-9}t^6}\]
where $f_{ij} = f_{ij}(q,t)$ is given by 
\begin{align*}
f_{11} &=  q^7t^{-5}+qt^{-1}, & f_{12} &= q^5t^{-4}+q^3t^{-2}+q, \allowdisplaybreaks \\
f_{13} &= q^4t^{-3}+q^2t^{-1}, &
f_{14} &= q^3t^{-2}, \allowdisplaybreaks \\
f_{22} &= q^7t^{-5}+(q^5+q^3)t^{-3}+(q^3+2q)t^{-1}, &
f_{23} &= q^6t^{-4}+(q^4+q^2)t^{-2}+1, \allowdisplaybreaks \\
f_{24} &= q^5t^{-3}+qt^{-1}, &
f_{33} &= q^8t^{-5}+(q^6+q^4)t^{-3}+(2q+1)t^{-1}, \allowdisplaybreaks \\
f_{34} &= q^7t^{-4}+q^3t^{-2}+q, &
f_{44} &= q^8t^{-5}+q^2t^{-1}.  
\end{align*}
For the other case $i > j$, we can use the relation $[d_i]_q \tC_{ij}(q,t) = [d_j]_q \tC_{ji}(q,t)$.

When $C$ is of type $\mathrm{ABCD}$, an explicit formula of $\tC(q,t)$ is given in \cite[Appendix C]{FR98}.
When $C$ is of type $\mathrm{ADE}$, we have $\tC(q,t) = \tC(qt^{-1},1)$ and an explicit formula of $\tC(q) = \tC(q, 1)$ is given in \cite[Appendix A]{GTL} (see also \cite[\S\S 4.4.1,4.4.2]{KO}).
\end{Rem}

\section{Generalized preprojective algebras}
\label{Sec:gpa}

Throughout this section, we fix an arbitrary field $\kk$. 
Unless specified otherwise, vector spaces and algebras are defined over $\kk$, and modules are left modules.

\subsection{Conventions}
Let $Q$ be a finite quiver.
We understand it as a quadruple $Q=(Q_0, Q_1, \se, \tl)$, where $Q_0$ is the set of vertices, $Q_1$ is the set of arrows and $\se$ (resp.~$\tl$) is the map $Q_1 \to Q_0$ which assigns each arrow with its source (resp.~target).
For a quiver $Q$, we set $\kk Q_0 \seq \bigoplus_{i \in Q_0} \kk e_i$ and $\kk Q_1 \seq \bigoplus_{\alpha \in Q_1} \kk \alpha$.  
We endow $\kk Q_0$ with a $\kk$-algebra structure by $e_i \cdot e_j = \delta_{ij} e_i$ for any $i,j \in Q_0$, and $\kk Q_1$ with a $(\kk Q_0, \kk Q_0)$-bimodule structure by $e_i \cdot \alpha = \delta_{i, \tl(\alpha)} \alpha$ and $\alpha \cdot e_i = \delta_{i, \se(\alpha)} \alpha$ for any $i \in Q_0$ and $\alpha \in Q_1$. 
Then the path algebra of $Q$ is defined to be the tensor algebra $\kk Q \seq T_{\kk Q_0}(\kk Q_1)$. 

Let $G$ be a multiplicative abelian group with unit $1$. 
By a $G$-graded quiver, we mean a quiver $Q$ equipped with a map $\deg \colon Q_1 \to G$. 
We regard its path algebra $\kk Q$ as a $G$-graded algebra in the natural way.
 
We say that a $G$-graded vector space $V = \bigoplus_{g \in G}V_g$ is locally finite if $V_g$ is of finite dimension for all $g \in G$.
In this case, we define its graded dimension $\dim_G V$ to be the formal sum $\sum_{g \in G}\dim_\kk (V_g) g \in \Z[\![G]\!]$.
For a $G$-graded vector space $V$ and an element $x \in G$, we define the grading shift $xV = \bigoplus_{g \in G}(xV)_{g}$ by $(xV)_g = V_{x^{-1}g}$.
More generally, for $a = \sum_{g \in G} a_g g \in \Z_{\ge 0}[\![G]\!]$, we set $V^{\oplus a} \seq \bigoplus_{g \in G} (gV)^{\oplus a_g}$. 
When $V^{\oplus a}$ happens to be locally finite, we have $\dim_G V^{\oplus a} = a \dim_G V$. 

\subsection{Preliminary on positively graded algebras} \label{Ssec:pga}
Let $t^{\Z}$ denote a free abelian group generated by a non-trivial element $t$.
In what follows, we consider the case when $G$ is a direct product $G = G_0 \times t^{\Z}$, where $G_0$ is another abelian group.
Our principal example is the group $\Gamma = t^\Z \times \Gamma_0$ in \S\ref{Ssec:DCM}.
For $G$-graded vector space $V = \bigoplus_{g \in G} V_g$ and $n \in \Z$, we define the $G_0$-graded subspace $V_n \subset V$ by $V_n \seq \bigoplus_{g \in G_0}V_{t^n g}$.
By definition, we have $V = \bigoplus_{n \in \Z}V_n$.
We use the notation $V_{\ge n} \seq \bigoplus_{m \ge n} V_m$ and $V_{>n} \seq \bigoplus_{m > n}V_m$. 

We consider a $G$-graded algebra $\Lambda$ satisfying the following condition:
\begin{itemize}
\item[(A)] $\Lambda = \Lambda_{\ge 0}$ and $\dim_\kk \Lambda_n < \infty$ for each $n \in \Z_{\ge 0}$.
\end{itemize}
In particular, $\Lambda_0$ is a $G_0$-graded finite dimensional algebra.
Let $\{ S_j\}_{j \in J}$ be a complete collection of $G_0$-graded simple modules of $\Lambda_0$ up to isomorphism and grading shift. 
It also gives a complete collection of $G$-graded simple modules of $\Lambda$.
For a $G$-graded $\Lambda$-module $M$, the subspace $M_{\ge n} \subset M$ is a $\Lambda$-submodule for each $n \in \Z$.
Let $\Lambda \umod_G^{\ge n}$ denote the category of $G$-graded $\Lambda$-modules $M$ satisfying $M = M_{\ge n}$ and $\dim_\kk M_m < \infty$ for all $m \ge n$, whose morphisms are $G$-homogeneous $\Lambda$-homomorphisms.
This is a $\kk$-linear abelian category.
Let $\Lambda \umod_G^+ \seq \bigcup_{n \in \Z} \Lambda \umod_G^{\ge n}$.  
Note that $\Lambda \umod_G^{+}$ contains all the finitely generated $G$-graded $\Lambda$-modules, because it contains their projective covers by the condition (A).

\begin{Lem} \label{Lem:PM}
Given $n \in \Z$ and $M \in \Lambda \umod_G^{\ge n}$, there is a surjection $P \twoheadrightarrow M$ from a projective $\Lambda$-module $P$ belonging to $\Lambda \umod_G^{\ge n}$ .
\end{Lem}
\begin{proof}
For each $m \ge n$, let $P_m \twoheadrightarrow M_m$ be a projective cover of $M_m$ regarded as a $G_0$-graded $\Lambda_0$-module.
Then consider the $G$-graded projective $\Lambda$-module 
$P \seq \Lambda \otimes_{\Lambda_0}\bigoplus_{m \ge n} t^mP_m$, which carries a natural surjection $P \twoheadrightarrow M$.
This $P$ belongs to $\Lambda \umod_G^{\ge n}$ because $\dim_G P$ is not greater than $\dim_{G} \Lambda \cdot \sum_{m \ge n}t^m \dim_{G_0} P_m$ which belongs to $\Z[G_0][\![t]\!] t^n$.
\end{proof}

For an abelian category $\mathcal{C}$, we denote by $K(\mathcal{C})$ its Grothendieck group.
We regard $K(\Lambda \umod_G^{\ge n})$ as a subgroup of $K(\Lambda \umod_G^+)$ via the inclusion for any $n \in \Z$.
Then, the collection of subgroups $\{ K(\Lambda \umod_G^{\ge n})\}_{n \in \Z}$ gives a  filtration of $K(\Lambda \umod_G^+)$.
We define the completed Grothendieck group $\hK(\Lambda \umod_G^+)$ to be the projective limit
\[ \hK(\Lambda \umod_G^+) \seq \varprojlim_{n} K(\Lambda \umod_G^{+})/K(\Lambda \umod_G^{\ge n}). \] 
Note that $\hK(\Lambda \umod_G^+)$ carries a natural $\Z[G_0](\!( t )\!)$-module structure given by $a [M] = [M^{\oplus a_+}] - [M^{\oplus a_-}]$, where we choose $a_+, a_- \in \Z_{\ge 0}[G_0](\!(t)\!)$ so that $a = a_+ - a_-$.
\begin{Lem} \label{Lem:Sbasis}
The $\Z[G_0](\!(t)\!)$-module $\hK(\Lambda \umod_G^+)$ is free with a basis $\{[S_j]\}_{j \in J}$.
\end{Lem}
\begin{proof}
For any $n \in \Z$ and $M \in \Lambda\umod^{\ge n}_G$, we have a unique expression 
\[[M] = \sum_{j\in J}\left(\sum_{m \geq n}[M_m:S_j]_{G_0}t^m\right)[S_j]\] 
in $\hK(\Lambda \umod_G^+)$, where $[M_m:S_j]_{G_0} \in \Z[G_0]$ denotes the $G_0$-graded Jordan-H\"older multiplicity of $S_j$ in the finite length $G_0$-graded $\Lambda_0$-module $M_m$.
This proves the assertion.
\end{proof}

\subsection{Generalized preprojective algebras}
\label{Ssec:GPA}

We fix a GCM $C = (c_{ij})_{i,j \in I}$ and its symmetrizer $D = \diag(d_i \mid i \in I)$ as in \S\ref{Ssec:notation}.
Recall the free abelian group $\Gamma$ in \S\ref{Ssec:DCM}.
We consider the quiver $\tQ = (\tQ_0, \tQ_1, \se, \tl)$  given as follows: 
\begin{gather*}
\tQ_0 = I, \quad
\tQ_1 = \{ \alpha_{ij}^{(g)} \mid (i,j, g) \in I \times I \times \Z, i \sim j, 1 \le g \le g_{ij}\} \cup \{ \ep_i \mid i \in I \}, \\
\se(\alpha_{ij}^{(g)}) = j, \quad \tl(\alpha_{ij}^{(g)})= i, \quad \se(\ep_i)=\tl(\ep_i)=i. 
\end{gather*}
We equip the quiver $\tQ$ with a $\Gamma$-grading by 
\begin{equation} \label{eq:deg}
\deg (\alpha_{ij}^{(g)}) = q^{-d_if_{ij}}t\mu^{(g)}_{ij}, \qquad \deg(\ep_i) = q^{2d_i}.
\end{equation} 
Let $\Omega$ be an acyclic orientation of $C$.
We define the associated potential $W_\Omega \in \kk \tQ$ by 
\begin{equation} \label{eq:potential} 
W_\Omega = \sum_{i, j \in I; i \sim j} 
\sum_{g=1}^{g_{ij}} \mathrm{sgn}_\Omega(i,j)\alpha_{ij}^{(g)} \alpha_{ji}^{(g)}\ep_{i}^{f_{ij}}, \end{equation}
where $\mathrm{sgn}_\Omega(i,j) \seq(-1)^{\delta((j,i) \in \Omega)}$. 
Note that $W_\Omega$ is homogeneous of degree $t^2$. 
We define the $\Gamma$-graded $\kk$-algebra $\tPi$ to be the quotient of $\kk\tQ$ by the ideal generated by all the cyclic derivations of $W_\Omega$.
In other words, the algebra $\tPi$ is the quotient of $\kk \tQ$ by the following two kinds of relations:
\begin{itemize} 
\item[(R1)] $\ep_i^{f_{ij}} \alpha_{ij}^{(g)} = \alpha_{ij}^{(g)} \ep_j^{f_{ji}}$ for any $i,j \in I$ with $i \sim j$ and $1 \le g \le g_{ij}$;
\item[(R2)] $\displaystyle \sum_{j \in I: j\sim i}\sum_{g=1}^{g_{ij}}\sum_{k = 0}^{f_{ij}-1}\mathrm{sgn}_\Omega(i,j)  \ep_i^k \alpha_{ij}^{(g)} \alpha_{ji}^{(g)} \ep_i^{f_{ij}-1-k} =0$ for each $i \in I$.
\end{itemize} 
\begin{Rem} \label{Rem:ori}
Although the definition of the algebra $\tPi$ depends on the choice of acyclic orientation $\Omega$, it is irrelevant. 
In fact, a different choice of $\Omega$ gives rise to an isomorphic $\Gamma$-graded algebra.  
Moreover, one may define $\tPi$ with more general orientation (i.e., without acyclic condition, as in \S\ref{Ssec:KP} below).
Even if we do so, the resulting $\Gamma$-graded algebra is isomorphic to our $\tPi$.    
\end{Rem}
For a positive integer $\ell \in \Z_{>0}$, we define the $\Gamma$-graded algebra $\Pi(\ell)$ to be the quotient
\[ \Pi(\ell) = \tPi/ (\ep^\ell), \]
where $\ep \seq \sum_{i \in I} \ep_i^{r/d_i}$. 
Note that $\ep$ is homogeneous and central in $\tPi$.
In other words, it is the quotient of $\kk \tQ$ by the three kinds of relations: (R1), (R2), and
\begin{itemize}
\item[(R3)] $\ep_i^{r\ell / d_i} = 0$ for each $i \in I$.
\end{itemize}  
\begin{Rem}\label{Rem:GLS}
    The algebra $\Pi(\ell)$ is identical to the generalized preprojective algebra $\Pi({}^\mathtt{t}C, \ell rD^{-1}, \Omega)$ in the sense of \cite{GLS}. 
\end{Rem}
\begin{Lem} \label{Lem:tpos}
For any $\ell \in \Z_{>0}$, the algebra $\Pi(\ell)$ satisfies the condition {\rm(A)} in {\rm \S\ref{Ssec:pga}}.  
\end{Lem}
\begin{proof}
The fact $\Pi(\ell)_{\ge 0} =\Pi(\ell)$ is clear from the definition \eqref{eq:deg}.
For any $n \in \Z_{\ge 0}$, 
thanks to the relation {\rm (R3)}, the vector space $\Pi(\ell)_n$ is spanned by a finite number of vectors in
\[ \{ \ep_{i_0}^{m_0} \alpha^{(g_1)}_{i_0,i_1} \ep_{i_1}^{m_1} \alpha^{(g_2)}_{i_1,i_2} \cdots \ep_{i_{n-1}}^{m_{n-1}}\alpha^{(g_{n})}_{i_{n-1},i_n}\ep_{i_n}^{m_n} \mid i_k \in I, 0 \le m_k < r \ell/d_{i_k}, 1 \le g_k \le g_{i_{k-1},i_k} \}.\]
Therefore, we have $\dim_\kk \Pi(\ell)_n < \infty$.
\end{proof}

In what follows, we fix $\ell \in \Z_{>0}$ and write $\Pi$ for $\Pi(\ell)$ for the sake of brevity.

By the definition, we have
\[ \Pi_0 \cong \prod_{i \in I} H_i, \qquad \text{where $H_i \seq \kk[\ep_i]/(\ep_i^{r\ell/d_i})$}. \]
In particular, for each $M \in \Pi \umod_\Gamma^+$ and $n \in \Z$, the subspace $e_i M_n$ is a finite-dimensional $H_i$-module for each $i \in I$.
We say that $M$ is locally free if $e_i M_n$ is a free $H_i$-module for any $n \in \Z$ and $i \in I$, or  equivalently $M_n$ is a projective $\Pi_0$-module for any $n \in \Z$.  
In this case, we set $\rank_{i} M \seq \dim_\Gamma e_i (M/\ep_i M) \in \Z[\Gamma_0](\!(t)\!)$.

\begin{Thm}[{\cite[\S 11]{GLS}}]
As a (left) $\Pi$-module, $\Pi$ is locally free in itself. 
\end{Thm}

For each $i \in I$, let $P_i \seq \Pi e_i$ be the indecomposable projective $\Pi$-module associated to the vertex $i$ and $S_i$ its simple quotient.  
Consider the two-sided ideal $J_i \seq \Pi (1-e_i) \Pi$. 
We have $\Pi / J_i \cong H_i$ as $\Gamma$-graded algebras.
We write $E_i$ for $\Pi / J_i$ when we regard it as a $\Gamma$-graded left $\Pi$-module.
This is a locally free $\Pi$-module characterized by $\rank_j E_i = \delta_{i,j}$.
In $\hK(\Pi\umod_\Gamma^+)$, we have 
\begin{equation} \label{eq:E=S}
[E_i] = \frac{1-q^{2r\ell}}{1-q^{2d_i}}[S_i].
\end{equation}

There is the anti-involution $\phi \colon \Pi \rightarrow \Pi^{\op}$ given by the assignment
\[\phi(e_i) \seq e_i, \qquad \phi(\alpha_{ij}^{(g)}) \seq \alpha_{ji}^{(g)}, \qquad \phi(\ep_i) \seq \ep_i.\]
Recall the automorphism of the group $\Gamma$ also denoted by $\phi$ in \S\ref{Ssec:Braid}. 
By definition, if $x \in \Pi$ is homogeneous of degree $\gamma \in \Gamma$, then $\phi(x)$ is homogeneous of degree $\phi(\gamma)$. 
For a left $\Pi$-module $M$, let $M^\phi$ be the right $\Pi$-module obtained by twisting the original left $\Pi$-module structure by the opposition $\phi$.
If $M$ is $\Gamma$-graded, $M^\phi$ is again $\Gamma$-graded by setting $(M^\phi)_\gamma \seq M_{\phi(\gamma)}$.
In particular, for $M \in \Pi \umod_\Gamma^+$, we have $\dim_\Gamma (M^\phi) = (\dim_\Gamma M)^\phi$.

\subsection{Projective resolutions}
Following \cite[\S 5.1]{GLS}, for each $i, j \in I$ with $i \sim j$, we define the bigraded $(H_i, H_j)$-bimodule ${}_i H_j$ by
\[ {}_i H_j  \seq \sum_{g=1}^{g_{ij}} H_i \alpha_{ij}^{(g)} H_j  \subset \Pi. \]
It is free as a left $H_i$-module and free as a right $H_j$-module. 
Moreover, the relation $(R1)$ gives the following:
\[ {}_i H_j = \bigoplus_{k=0}^{f_{ji}-1}\bigoplus_{g=1}^{g_{ij}}H_i(\alpha_{ij}^{(g)}\varepsilon_j^k) 
    = \bigoplus_{k=0}^{f_{ij}-1}\bigoplus_{g=1}^{g_{ij}}(\varepsilon_i^k \alpha_{ij}^{(g)}) H_j. \]
In particular, we get the following lemma.
\begin{Lem}\label{lem:iHjCartan}
For $i, j\in I$ with $i\sim j$, we have two isomorphisms
\[
{}_{H_i}({}_i H_j) \cong H_i^{\oplus (-q^{-d_j}tC_{ji}(q,t,\mmu)^\phi)}, \qquad
({}_i H_j) {}_{H_j} \cong H_j^{\oplus (-q^{-d_i}tC_{ij}(q,t, \mmu))}
\]
as $\Gamma$-graded left $H_i$-modules and as $\Gamma$-graded right $H_j$-modules respectively.
\end{Lem}

Consider the following sequence of $\Gamma$-graded $(\Pi, \Pi)$-bimodules: 
\begin{equation} \label{eq:bmodres}
 \bigoplus_{i\in I} q^{-2d_i} t^2 \Pi e_i \otimes_i e_i \Pi
 \xrightarrow{\psi}
 \bigoplus_{i, j \in I: i \sim j} \Pi e_j \otimes_j {}_jH_i \otimes_i e_i \Pi 
 \xrightarrow{\varphi}
 \bigoplus_{i\in I} \Pi e_i \otimes_i e_i \Pi \rightarrow \Pi \rightarrow 0,
\end{equation}
where $\otimes_i \seq \otimes_{H_i}$ and the morphisms $\psi$ and $\varphi$ are given by 
\begin{align*}
\psi(e_i \otimes e_i) &\seq \sum_{j \sim i}\sum_{k=0}^{f_{ij}-1}\sum_{g=1}^{g_{ij}} \mathrm{sgn}_\Omega(i,j) \left(\ep_i^{k} \alpha_{ij}^{(g)} \otimes \alpha_{ji}^{(g)} \ep_i^{f_{ij}-1-k}\otimes e_i + e_i \otimes \ep_i^{k} \alpha_{ij}^{(g)} \otimes \alpha_{ji}^{(g)} \ep_i^{f_{ij}-1-k}\right), \\
\varphi(e_j \otimes x\otimes e_i) &\seq x \otimes e_i + e_j \otimes x.
\end{align*}
The other arrows $\bigoplus_{i\in I} \Pi e_i \otimes_i e_i \Pi \to \Pi \to 0$ are canonical. 
The relation (R2) ensures that the sequence~\eqref{eq:bmodres} forms a complex. 
For each $i \in I$, applying $(-)\otimes_{\Pi} E_i$ to \eqref{eq:bmodres} yields the following complex of $\Gamma$-graded (left) $\Pi$-modules:
\begin{equation} \label{res3}
q^{-2d_i}t^2 P_i \xrightarrow{\psi^{(i)}}
\bigoplus_{j\sim i} P_j^{\oplus (-q^{-d_i}tC_{ij}(q,t,\mmu)^\phi)}
\xrightarrow{\varphi^{(i)}}
P_i \to E_i \to 0.
\end{equation}
Here we used Lemma~\ref{lem:iHjCartan}.

\begin{Thm}[{\cite[Proposition 12.1 and Corollary 12.2]{GLS}, \cite[Theorem 3.3]{FM}}]
\label{Thm:Pres}
The complexes \eqref{eq:bmodres} and \eqref{res3} are exact.
Moreover, the followings hold.
\begin{enumerate}
\item When $C$ is of infinite type, we have $\Ker \psi = 0$ and $\Ker \psi^{(i)} =0$ for all $i \in I$.
\item When $C$ is of finite type, we have $\Ker \psi^{(i)} \cong q^{-rh^\vee}t^h\mu_{i^*i} E_{i^*}$ for each $i \in I$.
\end{enumerate}
\end{Thm}

\subsection{Euler-Poincar\'{e} pairing} 
\label{Ssec:EP}
For a $\Gamma$-graded right $\Pi$-module $M$ and a $\Gamma$-graded left $\Pi$-module $N$, the vector space $M \otimes_\Pi N$ is naturally $\Gamma$-graded.
Let $\gtor_k^\Pi (M, N)$ 
denote the $k$-th left derived functor of $M \mapsto M \otimes_\Pi N$ (or equivalently, that of $N \mapsto M \otimes_\Pi N$). 

\begin{Lem} \label{Lem:tor}
If $M \in \Pi^\op \umod_\Gamma^{\ge m}$ and $N \in \Pi \umod_\Gamma^{\ge n}$,  we have
\[ \gtor_k^\Pi(M, N) \in  \kk \umod_\Gamma^{\ge m+n} \quad \text{for any $k \in \Z_{\ge 0}$}.\]
\end{Lem}
\begin{proof}
We see that $\dim_\Gamma (M\otimes_\Pi N)$ is not grater than $\dim_\Gamma M \cdot \dim_\Gamma N$ which belongs to $\Z[\Gamma_0][\![t]\!]t^{m+n}$ under the assumption. 
This proves the assertion for $k =0$.
The other case when $k > 0$ follows from this case and Lemma~\ref{Lem:PM}.  
\end{proof}

We consider the following finiteness condition for a pair $(M,N)$ of objects in $\Pi \umod_\Gamma^+$:
\begin{itemize}
\item[(B)]
For each $\gamma \in \Gamma$, the space $\gtor^\Pi_k(M^\phi, N)_\gamma$ vanishes for $k \gg 0$.
\end{itemize}
If $(M,N)$ satisfies the condition (B), their \emph{Euler-Poincar\'{e} (EP) pairing} 
\[ \langle M,N \rangle_\Gamma \seq \sum_{k =0}^{\infty} (-1)^k \dim_\Gamma \gtor^\Pi_k(M^\phi, N).\]
is well-defined as an element of $\Z[\![\Gamma ]\!]$. 
The next lemma is immediate from the definition.

\begin{Lem} \label{Lem:EP}
Let $M, N \in \Pi\umod_\Gamma^+$.
\begin{enumerate}
\item If $(M,N)$ satisfies {\rm (B)}, the opposite pair $(N,M)$ also satisfies {\rm (B)} and we have $\langle N,M \rangle_\Gamma = \langle M, N \rangle_\Gamma^\phi$.
\item If $(M,N)$ satisfies {\rm (B)}, the pair $(M^{\oplus a}, N^{\oplus b})$ also satisfies {\rm (B)} for any $a, b \in \Z_{\ge 0}[\Gamma]$ and we have 
$\langle M^{\oplus a}, N^{\oplus b} \rangle_\Gamma = a^\phi b \langle M, N \rangle_\Gamma$.
\item \label{Lem:EP:3} Suppose that there is an exact sequence $0 \to M' \to M \to M'' \to 0$ in $\Pi \umod_\Gamma^+$, and the pairs $(M', N)$ and $(M'', N)$ both satisfy {\rm (B)}.
Then the pair $(M,N)$ also satisfies {\rm (B)} and we have $\langle M, N\rangle_\Gamma = \langle M', N \rangle_\Gamma + \langle M'', N \rangle_\Gamma.$\qedhere
\end{enumerate}
\end{Lem}

\begin{Prop} \label{Prop:EP}
For any $i,j \in I$, the pair $(S_i, S_j)$ satisfy the condition {\rm (B)} and we have
\begin{align} \label{eq:ES}
\langle E_i, S_j \rangle_\Gamma &= \begin{cases} 
\displaystyle
\frac{q^{-d_i}t \left( C_{ij}(q,t,\mmu) - q^{-rh^\vee}t^{h}\mu_{ii^*} C_{i^*j}(q,t,\mmu) \right)}{1-q^{-2rh^\vee}t^{2h}}& \text{if $C$ is of finite type}, \\
q^{-d_i}tC_{ij}(q,t,\mmu) & \text{otherwise}, 
\end{cases}
\\
\label{eq:SS}
\langle S_i, S_j\rangle_\Gamma &= \frac{1-q^{2d_i}}{1-q^{2r\ell}} \langle E_i, S_j \rangle_\Gamma.
\end{align}
Here we understand $(1-\gamma)^{-1} = \sum_{k \ge 0}\gamma^{k} \in \Z[\![\Gamma]\!]$ for $\gamma \in \Gamma \setminus \{1\}$. 
\end{Prop}
\begin{proof}
The former formula \eqref{eq:ES} directly follows from Theorem~\ref{Thm:Pres}.
The latter \eqref{eq:SS} follows from Theorem~\ref{Thm:Pres} and the fact that $S_i$ has an $E_i$-resolution of the form: 
\[ \cdots \to q^{2r\ell + 2d_i}E_i \to q^{2r\ell} E_i \to q^{2d_i}E_i \to E_i \to S_i \to 0. \]
See the proof of \cite[Proposition 3.11]{FM} for some more details.
\end{proof}

\begin{Cor}\label{Cor:fin}
For any $M,N \in \Pi\umod_\Gamma^+$, the pair $(M,N)$ satisfies the condition {\rm (B)}.
Moreover, the EP pairing induces a $\phi$-sesquilinear hermitian form on the $\Z[\Gamma_0](\!(t)\!)$-module $\hK(\Pi \umod_\Gamma^+)$ valued at $\Z[\Gamma_0][(1-q^{2r\ell})^{-1}](\!(t)\!)$.
\end{Cor}
\begin{proof}
Given $M,N \in \Pi\umod_\Gamma^+$, we shall show that $(M,N)$ satisfies the condition {\rm (B)}.
Without loss of generality, we may assume that $M, N \in \Pi \umod_\Gamma^{\ge 0}$.
For $\gamma \in \Gamma$ fixed, take $n \in \mathbb{Z}$ such that $\gamma \not \in \Gamma_0 t^{\mathbb{Z}_{> n}}$.
By Lemma~\ref{Lem:tor}, we have $\gtor_k^\Pi(M^\phi_{>n},N)_{\gamma}  =0$  and therefore $\gtor_k^\Pi(M^\phi,N)_\gamma \simeq \gtor_k^\Pi(M^\phi/M^\phi_{>n},N)_\gamma$ for any $k \in \mathbb{Z}_{\ge 0}$. Similarly, we have $\gtor_k^\Pi(M^\phi/M^\phi_{>n},N)_\gamma \simeq \gtor_k^\Pi(M^\phi/M^\phi_{>n},N/N_{>n})_\gamma$ and hence $\gtor_k^\Pi(M^\phi,N)_\gamma \simeq \gtor_k^\Pi(M^\phi/M^\phi_{>n},N/N_{>n})_\gamma$ for any $k \in \mathbb{Z}_{\ge 0}$. 
By Lemma~\ref{Lem:EP} and Proposition~\ref{Prop:EP}, we know that the condition {\rm (B)} is satisfied for any finite-dimensional modules. 
Therefore, for $k$ large enough, we have $\gtor_k^\Pi(M^\phi/M^\phi_{>n},N/N_{>n})_\gamma = 0$. 
Thus, the pair $(M,N)$ satisfies the condition {\rm (B)}.
Now, by Lemma~\ref{Lem:EP}~\eqref{Lem:EP:3} and Proposition~\ref{Prop:EP}, the EP pairing induces a pairing on the Grothendieck group $K(\Pi \umod_G^+)$ valued at $\mathbb{Z}[\Gamma_0][(1-q^{2r\ell})^{-1}](\!(t)\!)$.
Lemma~\ref{Lem:tor} tells us that this is continuous with respect to the topology given by the filtration $\{ K(\Pi \umod_G^{\ge n})\}_{n \in \mathbb{Z}}$. 
Therefore, it descends to a pairing on the completion $\hK(\Pi \umod_\Gamma^+)$ satisfying the desired properties. 
\end{proof}

Let $\F$ be an algebraic closure of the field $\Q(\Gamma_0)(\!( t )\!)$. 
We understand that $\Q(\Gamma)$ is a subfield of $\F$ by considering the Laurent expansions at $t = 0$.
By Corollary~\ref{Cor:fin} above, the EP pairing linearly extends to a $\phi$-sesquilinear hermitian form, again written by   
$\langle - , - \rangle_\Gamma$, on 
\[ \hK(\Pi \umod_\Gamma^+)_\F \seq \hK(\Pi \umod_\Gamma^+)\otimes_{\Z[\Gamma_0](\!(t)\!)} \F \]
valued at $\F$.
Note that the set $\{ [E_i] \}_{i \in I}$ forms an $\F$-basis of $\hK(\Pi \umod_\Gamma^+)_\F$ by Lemma~\ref{Lem:Sbasis} and \eqref{eq:E=S},
and that, if $M \in \Pi \umod_\Gamma^+$ is locally free, we have $[M] = \sum_{i \in I}(\rank_i M) [E_i]$.  

It is useful to introduce the module $\bar{P}_i \seq P_i/P_i \ep_i = (\Pi/\Pi\ep_i)e_i$ for each $i \in I$. 
We can easily prove the following (see~\cite[Lemma 2.5]{FM}).

\begin{Lem} \label{Lem:barP}
If $M \in \Pi \umod_\Gamma^+$ is locally free, we have $\langle \bar{P}_i, M \rangle_\Gamma = \rank_i M$.
In particular, we have $\langle \bar{P}_i, E_j \rangle_\Gamma = \delta_{i,j}$ and $\rank_i P_j = (\dim_\Gamma e_j \bar{P}_i)^\phi$ for any $i,j \in I$.
\end{Lem}

On the other hand, we consider the $\F$-vector space $\sQ_\Gamma \otimes_{\Q(\Gamma)}\F$, on which the pairing $(-,-)_\Gamma$ extends linearly.
Let $\Psi$ be the $\F$-linear automorphism of $\sQ_\Gamma \otimes_{\Q(\Gamma)}\F$ given by 
\[ \Psi \seq \begin{cases}
\displaystyle
(1+ q^{-rh^\vee}t^h \nu)^{-1} =  
\frac{\id-q^{-rh^\vee}t^h\nu}{1-q^{-2rh^\vee}t^{2h}}  & \text{if $C$ is of finite type}, \\
\id & \text{otherwise}.
\end{cases}
\]
Here $\nu$ is the linear operator on $\sQ_\Gamma$ given by $\nu(\alpha_i) = \mu_{i^*i}\alpha_{i^*}$, which we have already defined in \S\ref{Ssec:Braid}. 
Let us choose an element $\kappa_\ell \in \F$ satisfying 
$
\kappa_\ell^2 = 
q^{r\ell}[r\ell]_qt^{-1}$.
\begin{Thm} \label{Thm:chi}
The assignment $[E_i] \mapsto \kappa_\ell \alpha_i^\vee \, (i \in I)$ gives an $\F$-linear isomorphism
\[ \chi_\ell \colon \hK(\Pi \umod_\Gamma^+)_\F  \to \sQ_\Gamma \otimes_{\Q(\Gamma)} \F\]
satisfying the following properties:
\begin{enumerate}
\item \label{Thm:chi:1}
For any $i \in I$, we have $\chi_\ell[S_i] = \kappa_\ell^{-1} \alpha_i$.
\item \label{Thm:chi:2}
For any $x,y \in \hK(\Pi \umod_\Gamma^+)_\F$, we have 
$\langle x,y \rangle_\Gamma = 
( \Psi\chi_\ell(x), \chi_\ell(y))_\Gamma.$
\item \label{Thm:chi:3}
For any $i \in I$, we have
$\varpi_i^\vee = \kappa_\ell^{-1}\Psi\chi_\ell[P_i]$ and $\varpi_i = q^{-d_i}t\kappa_\ell \Psi \chi_\ell[\bar{P}_i]$.
\end{enumerate}
\end{Thm}
\begin{proof}
As the set $\{ [E_i]\}_{i \in I}$ forms an $\F$-basis of $\hK(\Pi \umod_\Gamma^+)_\F$, the linear map $\chi_\ell$ is an isomorphism. 
The properties \eqref{Thm:chi:1} \& \eqref{Thm:chi:2} follow from the identities \eqref{eq:E=S}, \eqref{eq:ES} and \eqref{eq:SS}. 
Since the basis $\{[P_i]\}_{i \in I}$ (resp.~$\{[\bar{P}_i]\}_{i \in I}$) is dual to the basis $\{[S_i]\}_{i \in I}$ (resp.~$\{[E_i]\}_{i \in I}$ by Lemma~\ref{Lem:barP}), the property \eqref{Thm:chi:3} follows from the property~\eqref{Thm:chi:2}. 
\end{proof}

\begin{Cor} \label{Cor:chi}
Let $i,j \in I$.
\begin{enumerate}
\item
When $C$ is of finite type, we have 
\[ 
\tC_{ij}(q,t,\mmu) = 
\frac{q^{-d_i}t}{1-q^{-2rh^\vee}t^{2h}} \left(\dim_{\Gamma}(e_i \bar{P}_j) - q^{-rh^\vee}t^h \mu_{ii^*}\dim_{\Gamma}(e_{i^*} \bar{P}_j) \right).
\]
\item \label{Cor:chi:2} When $C$ is of infinite type, we have
\[
\tC_{ij}(q,t,\mmu) = 
q^{-d_j}t\dim_{\Gamma}(e_i \bar{P}_j). \]
\end{enumerate}
\end{Cor}
\begin{proof}
It follows from Theorem~\ref{Thm:chi} and the inversion of \eqref{eq:ao}.
\end{proof}

In particular, Corollary~\ref{Cor:chi}~\eqref{Cor:chi:2} proves Theorem~\ref{Thm:tCpos}. 
\begin{Rem}
Since $\langle P_i, \bar{P}_j \rangle_{\Gamma} = \dim_\Gamma (e_i \bar{P}_j)$, Corollary~\ref{Cor:chi} interprets the matrix $\tC(q,t,\mmu)$ in terms of the EP pairing between the bases $\{[P_i]\}_{i \in I}$ and $\{[\bar{P}_i]\}_{i \in I}$.
In this sense, Corollary~\ref{Cor:chi} is dual to \eqref{eq:ES} in Proposition~\ref{Prop:EP}.
\end{Rem}
\begin{Rem}
In the previous paper~\cite{FM}, we dealt with GCMs of finite type and finite dimensional $(q,t)$-graded $\Pi$-modules. Therein, we used the modules $\bar{I}_i \seq \bD(\bar{P}_i^\phi)$ and the graded extension groups $\gext_\Pi^k$, where $\bD$ is the graded $\kk$-dual functor, instead of the modules $\bar{P}_i$ and the graded torsion groups $\gtor^\Pi_k$.
Note that the two discussions are mutually equivalent thanks to the usual adjunction (cf.~\cite[\S A.4 Proposition~4.11]{ASS}), i.e., we have $\bD(\gtor_k^\Pi(\bD(M),N)) \simeq \gext^k_\Pi(N,M)$ for $M,N \in \Pi \umod_\Gamma^+$. 
In this sense, our discussion here is a slight generalization of that in \cite{FM} with the additional $\mmu$-grading.
\end{Rem}

\subsection{Braid group action}\label{sec:braid}

Recall the two-sided ideal $J_i = \Pi (1-e_i) \Pi$.
For any $M \in \Pi \umod_\Gamma^+$ and $k \in \Z_{\ge 0}$, we see that $\gtor_k^\Pi (J_i, M )$ also belongs to $\Pi \umod_\Gamma^+$.
When $C$ is of infinite type, $J_i^{\phi} = J_i$ has projective dimension at most $1$. In particular, the derived tensor product $J_i \Lotimes_\Pi M$ is an object in the bounded derived category $\mathcal{D}^b (\Pi \umod_\Gamma^+)$ for each $M \in \Pi \umod_\Gamma^+$.
By the natural identification $K(\mathcal{D}^b (\Pi \umod_\Gamma^+))\cong K(\Pi \umod_\Gamma^+)$ and the canonical map $K(\Pi \umod_\Gamma^+)\rightarrow \hK(\Pi\umod_\Gamma^+)$, it gives the element  
\begin{equation} \label{eq:Lotimes}
[J_i \Lotimes_\Pi M] = \sum_{k = 0}^{\infty}(-1)^k[\gtor^\Pi_k(J_i, M)] = \sum_{j \in I}\langle J_i e_j, M \rangle_\Gamma [S_j]
\end{equation}
of $\hK(\Pi\umod_\Gamma^+)$, where the second equality follows since $J_i^\phi = J_i$.
When $C$ is of finite type, we define the element $[J_i \Lotimes_\Pi M]$ of $\hK(\Pi\umod_\Gamma^+)$ by \eqref{eq:Lotimes}.
Recalling the relation $E_i = \Pi/J_i$, we have $[J_ie_j] = [P_j] -\delta_{i,j}[E_i]$
for each $j \in I$, and hence
\[ [J_i \Lotimes_\Pi M] = [M] - \langle E_i, M \rangle_\Gamma [S_i]. \]
Sending this equality by the isomorphism $\chi_\ell$ in Theorem~\ref{Thm:chi}, we get
\begin{equation} \label{eq:JiM}
\chi_\ell[J_i \Lotimes_\Pi M] = \chi_\ell[M]-(\Psi \alpha_i^\vee, \chi_\ell[M])_\Gamma \alpha_i.
\end{equation}
In particular, we obtain the following analogue of \cite[Proposition 2.10]{AIRT}.
\begin{Lem} \label{Lem:JT}
When $C$ is of infinite type, we have 
\[\chi_\ell[J_i \Lotimes_\Pi M] = T_i \chi_\ell[M] \qquad \text{for any $M \in \Pi \umod_\Gamma^+$ and $i \in I$}.\]
\end{Lem}
\begin{proof}
When $C$ is of infinite type, we have $\Psi = \id$ by definition. 
Thus, the equation~\eqref{eq:JiM} coincides with the defining equation~\eqref{eq:Troot} of $T_i$ in this case. 
\end{proof}

\begin{proof}[Proof of Proposition~\ref{Prop:braidrel}]
When $C$ is of finite type, we can reduce the proof to the case of affine type since the collection $\{ T_i \}_{i \in I}$ can be extended to the collection $\{T_i\}_{i \in I\cup \{0\}}$ of the corresponding untwisted affine type.  
Hence, it suffices to consider the case when $C$ is of infinite type. 
In this case, the braid relations for  $\{ T_i \}_{i \in I}$ follow from Lemma~\ref{Lem:JT} and the fact that the ideals $\{J_i\}_{i \in I}$ satisfy the braid relations with respect to multiplication, which is due to \cite[Theorem~4.7]{FG}.  
For example, when $c_{ij}c_{ji}=1$, we have 
\[ J_i \Lotimes_\Pi J_j \Lotimes_\Pi J_i \simeq J_i \otimes_\Pi J_j \otimes_\Pi J_i \simeq J_iJ_jJ_i = J_jJ_iJ_j \simeq J_j \otimes_\Pi J_i \otimes_\Pi J_j \simeq J_j \Lotimes_\Pi J_i \Lotimes_\Pi J_j, \]
which implies the desired braid relation $T_iT_jT_i = T_jT_iT_j$.
\end{proof}

\begin{Cor}\label{Cor:braid:ac}
Let $M \in \Pi \umod_\Gamma^+$ with $\gtor_1^\Pi (J_i, M)=0$ for $i \in I$.
We have
\[\chi_\ell[J_i \otimes_{\Pi} M] = T_i\chi_\ell[M]. \]
Moreover, if we assume that $M$ is locally free, so is $J_i \otimes_{\Pi} M$.
\end{Cor}
\begin{proof}
     The first assertion is a direct consequence of Lemma~\ref{Lem:JT}.
     For $C$ of infinite type, since projective dimension of $J_i$ is at most $1$ by Theorem~\ref{Thm:Pres},
    our involution $\phi$ yields that $\gtor_k^\Pi (J_i, M)=0$ also for $k \geq 2$. This shows $[J_i \otimes_{\Pi} M]= [J_i \Lotimes_{\Pi} M]$ when $C$ is of infinite type. When $C$ is of finite type, our assertion follows easily from the exact embedding to the corresponding untwisted affine type $\hat{\Pi}$. Namely, we have an isomorphism $J_i\otimes_{\Pi} M \simeq \hat{J}_i \otimes_{\hat{\Pi}} M$, where $\hat{J}_i \coloneqq \hat{\Pi}(1-e_i)\hat{\Pi}$. The last assertion is just an analogue of \cite[Proof of Proposition~9.4]{GLS}.
\end{proof}
\begin{proof}[Proof of Theorem~\ref{Thm:inv1} when $C$ is of infinite type]
Assume that $C$ is of infinite type. 
Let $(i_k)_{k \in \Z_{>0}}$ be a sequence in $I$ satisfying the condition (2) in Theorem~\ref{Thm:inv1}.
We have a filtration $\Pi = F_0 \supset F_1 \supset F_2 \supset \cdots$ of $(\Pi,\Pi)$-bimodules given by $F_k \seq J_{i_1}J_{i_2} \cdots J_{i_k}$.
This filtration $\{F_k\}_{k \ge 0}$ is exhaustive, i.e., $\bigcap_{k \ge 0}F_k = 0$.
Indeed, since the algebra $\Pi$ satisfies the condition~$\mathrm{(A)}$, its radical filtration $\{ R_k\}_{k \ge 0}$ as a right $\Pi$-module is exhaustive. 
Note that, for any right $\Pi$-module $M$ and $i \in I$, the right module $M/MJ_i$ is the largest quotient of $M$ such that $(M/MJ_i)e_j \neq 0$ for $j \neq i$. 
Thanks to this fact and our assumption on the sequence $(i_k)_{k \in \Z_{>0}}$, 
we can find for each $k > 0$ a large integer $K$ such that $F_K \subset R_k$.
Thus, we have $\bigcap_{k} F_k = \bigcap_k R_k = 0$.
Moreover, by \cite[Proposition~3.8]{M2}, we have
\[F_{k-1}/F_{k} \simeq J_{i_1} J_{i_2}\cdots J_{i_{k-1}}\otimes_\Pi E_{i_k} \qquad \text{as $\Gamma$-graded left $\Pi$-modules}\]
for each $k \ge 1$.
Note that we have an equality $\gtor_1^\Pi (J_{i_1}, J_{i_2}\cdots J_{i_{k-1}}\otimes_\Pi E_{i_k})=0$ by \cite[Proof of Proposition~3.8]{M2}.
This yields $\chi_\ell [J_{i_1}J_{i_2}\cdots J_{i_{k-1}}\otimes_\Pi E_{i_k}] = \kappa_\ell T_{i_1}, T_{i_2}\cdots T_{i_{k-1}} \alpha^{\vee}_{i_k}$ inductively by Corollary~\ref{Cor:braid:ac}.

The filtration $\{F_k\}_{k \ge 0}$ induces an exhaustive filtration $\{ F_k e_i \}_{k \ge 0}$ of the projective module $P_i$ such that
\[ F_{k-1}e_i / F_{k}e_i \simeq 
\begin{cases}
J_{i_1} \cdots J_{i_{k-1}}\otimes_\Pi E_{i} & \text{if $i_k = i$}, \\
0 & \text{otherwise}
\end{cases}\]
for each $k \ge 1$.
Therefore, in $\hK(\Pi \umod_\Gamma^+)$, we have
\[ [P_i] = \sum_{k=1}^{\infty}[F_{k-1}e_i/F_{k}e_i] = \sum_{k \colon i_k = i}[J_{i_1} \cdots J_{i_{k-1}}\otimes_\Pi E_{i}].\]
Applying the isomorphism $\chi_\ell$ in Theorem~\ref{Thm:chi} to this equality, we obtain 
\[ \varpi_i^\vee = \sum_{k \colon i_k = i}T_{i_1}\cdots T_{i_{k-1}} \alpha_i^\vee.\] 
This is rewritten as 
\[ \varpi_i = q^{-d_i}t \sum_{k \colon i_k = i}T_{i_1}\cdots T_{i_{k-1}} \alpha_i. \]
Since $\tC_{ij}(q,t,\mmu) = (\varpi_i^\vee, \varpi_j)_\Gamma$ by \eqref{eq:ao}, we obtain the desired equality~\eqref{eq:inv1} from this.
\end{proof}
\begin{Rem}
In \cite{IR}, they proved that the ideal semigroup $\langle J_i \mid i \in I \rangle$ gives the set of isoclasses of classical tilting $\Pi$-modules for any symmetric affine type $C$ with $D = \id$.
In our situation, our two-sided ideals are $\Gamma$-graded tilting objects whose $\Gamma$-graded endomorphism algebras are isomorphic to $\Pi$ when $C$ is of infinite type by arguments in \cite{BIRS,FG}.
In particular, our braid group symmetry on $\hK(\Pi\umod_\Gamma^+)$ is induced from auto-equivalences on the derived category (cf.~\cite[\S 2]{MP}).
\end{Rem}

\section{Remarks}
\label{Sec:Rem}

\subsection{Comparison with Kimura-Pestun's deformation} \label{Ssec:KP}

In their study of \emph{(fractional) quiver $\mathcal{W}$-algebras}, Kimura-Pestun~\cite{KP2} introduced a deformation of GCM called \emph{the mass-deformed Cartan matrix}.
In this subsection, we compare their mass-deformed Cartan matrix with our deformed GCM $C(q,t,\mmu)$.

Let $Q$ be a quiver without loops and $d \colon Q_0 \to \Z_{>0}$ be a function. 
Following~\cite[\S2.1]{KP2}, we call such a pair $(Q,d)$ \emph{a fractional quiver}. 
We set $d_i \seq d(i)$ and $d_{ij} \seq \gcd(d_i,d_j)$ for $i,j \in Q_0$.
Let $C = (c_{ij})_{i,j \in I}$ be a GCM. 
We say that a fractional quiver $(Q,d)$ is of type $C$ if $Q_0 = I$ and the following condition is satisfied:
\begin{equation} \label{eq:fQ}
c_{ij} = 2\delta_{i,j} - (d_j/d_{ij})|\{ e \in Q_1 \mid \{\se(e), \tl(e)\} = \{ i,j \}\}| \quad \text{for any $i,j \in I$}.\end{equation}
In this case, $D = \diag(d_i \mid i \in I)$ is a symmetrizer of $C$ and we have 
\[
g_{ij} = |\{ e \in Q_1 \mid \{\se(e), \tl(e)\} = \{ i,j \}\}|, \quad f_{ij} = d_j/d_{ij} \quad \text{when $i \sim j$.} 
\]
See \S\ref{Ssec:notation} for the definitions.
For a given fractional quiver $(Q,d)$ of type $C$, Kimura-Pestun introduced a matrix $\CKP = (\CKP_{ij})_{i,j \in I}$, whose $(i,j)$-entry $\CKP_{ij}$ is a Laurent polynomial in the formal parameters $q_1, q_2$ and $\mu_e$ for each $e \in Q_1$ given by 
\begin{equation} \label{eq:CKP}
\CKP_{ij} \seq \delta_{i,j} (1+q_1^{-d_i}q_2^{-1}) - \frac{1-q_1^{-d_j}}{1-q_1^{-d_{ij}}}\left( \sum_{e \colon i \to j}\mu_e^{-1} + \sum_{e \colon j \to i}\mu_e q_1^{-d_{ij}}q_2^{-1}\right). \end{equation}
The parameters $\mu_e$ are called \emph{mass-parameters}.
If we evaluate all the parameters to $1$, the matrix $\CKP$ coincides with the GCM $C$ by \eqref{eq:fQ}.

Now we fix a function $g \colon Q_1 \to \Z_{>0}$ whose restriction induces a bijection between $\{ e \in Q_1 \mid \{\se(e), \tl(e)\} = \{ i,j \}\}$ and $\{ g \in \Z \mid 1 \le g \le g_{ij} \}$ for each $i,j \in I$ with $i \sim j$. 
Then consider the monomial transformation $\Z[q_1^{\pm 1},q_2^{\pm 1},\mu_e^{\pm 1} \mid e \in Q_1] \to \Z[\Gamma]$ given by 
\begin{equation} \label{eq:mt}
q_1 \mapsto q^2, \qquad 
q_2 \mapsto t^{-2}, \qquad
\mu_e \mapsto q^{d_{ij}}t^{-1}\mu_{ij}^{(g(e))},
\end{equation}
where $i = \tl(e)$ and $j = \se(e)$.
Note that it induces an isomorphism if we formally add the square roots of $q_1$ and $q_2$.
Under this monomial transformation, for any $i,j \in I$, we have
\begin{equation} \label{eq:mtCKP} \CKP_{ij} \mapsto q^{-d_j}t\left(
\delta_{i,j} (q^{d_i}t^{-1} + q^{-d_i}t) - \delta(i \sim j)
[f_{ij}]_{q^{d_{ij}}} \sum_{g = 1}^{g_{ij}}\mu_{ij}^{(g)}
\right). 
\end{equation} 

\begin{Prop}
Under the monomial transformation \eqref{eq:mt}, the matrix $\CKP$ corresponds to the matrix $C(q,t,\mmu)q^{-D}t$ if and only if the following condition is satisfied: \begin{equation} \label{eq:condf}
\text{For any $i,j \in I$ with $i \sim j$, we have $f_{ij} = 1$ or $f_{ji}=1$.}
\end{equation}  
\end{Prop}
\begin{proof}
Compare \eqref{eq:mtCKP} with \eqref{eq:defC} and note that we have $[f_{ij}]_{q^{d_{ij}}} = [f_{ij}]_{q^{d_i}}$ for any $i,j \in I$ with $i \sim j$ if and only if the condition~\eqref{eq:condf} is satisfied.
\end{proof}
\begin{Ex}
If we take our GCM $C$ and its symmetrizer $D$  as
    \[
C = \begin{pmatrix}2&-6\\-9&2\end{pmatrix} \quad \text{and} \quad D=\diag(3, 2)
,
\]
not satisfying \eqref{eq:condf},
then the image of $\CKP$ under \eqref{eq:mt} is 
\[ \begin{pmatrix} 1+q^{-6}t^{2} & -(q^{-1}+q^{-3})t(\mu_{12}^{(1)}+\mu_{12}^{(2)}+\mu_{12}^{(3)}) \\ -(q^{-1}+q^{-3}+q^{-5})t(\mu_{21}^{(1)}+\mu_{21}^{(2)}+\mu_{21}^{(3)}) & 1+q^{-4}t^2 \end{pmatrix},\]
which is different from 
\[C(q, t, \mmu) q^{-D}t = \begin{pmatrix}
    1+q^{-6}t^2 & -(q + q^{-5})t(\mu_{12}^{(1)}+\mu_{12}^{(2)}+\mu_{12}^{(3)})\\ -(q + q^{-3} + q^{-7})t(\mu_{21}^{(1)}+\mu_{21}^{(2)}+\mu_{21}^{(3)}) & 1 + q^{-4}t^2
\end{pmatrix}.\]
\end{Ex}
\begin{Rem} 
The condition~\eqref{eq:condf} is satisfied for all finite and affine types. 
It is also satisfied when $C$ is symmetric (i.e., ${}^\mathtt{t}C = C$).
This \eqref{eq:condf} also appears in \cite[\S C(iv)]{NW} as a condition for two possible mathematical definitions of Coulomb branches of quiver gauge theories with symmetrizers to coincide with each other as schemes.
\end{Rem}
\begin{Rem} \label{Rem:qC}
When \eqref{eq:condf} is satisfied, we can assure that the evaluation at $t = 1$ makes sense in the inversion formulas. 
More precisely, assuming~\eqref{eq:condf}, we see that the matrix $X$ in \eqref{eq:X} is written in the form $X = q^{-1}X'$ with $X'$ being a $\Z[\mmu^\Z][q^{-1},t]$-valued matrix (see the proof of \cite[Lemma 4.3]{FO21}), and hence we have $\tC_{ij}(q,t,\mmu) \in \Z[\mmu^\Z][q^{-1},t][\![(q^{-1}t)]\!]$ for any $i,j \in I$.
Thus, under~\eqref{eq:condf}, the evaluation at $t=1$ gives a well-defined element $\tC_{ij}(q,1,\mmu)$ of $\Z[\mmu^{\Z}][\![q^{-1}]\!]$. 
\end{Rem}

 \subsection{Universality of the grading}\label{subsec:univ_grad}
 
In this subsection, we briefly explain how one can think that our grading \eqref{eq:deg} on the algebra $\tPi$ is universal. 
It is stated as follows.

We keep the notation in \S\ref{Ssec:GPA}.
Let $\tG$ be the (multiplicative) abelian group generated by the finite number of formal symbols $\{[a] \mid a \in \tQ_1 \}$ subject to the relations
\begin{equation} \label{eq:Whom}
 [\alpha_{i_1j_1}^{(g_1)}] [\alpha_{j_1i_1}^{(g_1)}][\ep_{i_1}]^{f_{i_1j_1}} = 
[\alpha_{i_2j_2}^{(g_2)}] [\alpha_{j_2i_2}^{(g_2)}][\ep_{i_2}]^{f_{i_2j_2}}\end{equation}
for any $i_k,j_k \in I$ with $i_k \sim j_k$ and $1 \le g_k \le g_{i_kj_k}$ ($k = 1,2$). 
Let $\tG \twoheadrightarrow \tG_f$ be the quotient by the torsion subgroup. 
By construction, for any free abelian group $G$, giving a homomorphism $\deg \colon \tG_f \to G$ is equivalent to giving $\tQ$ a structure of $G$-graded quiver $\deg \colon \tQ_1 \to G$ such that the potential $W_\Omega$ is homogeneous. 
In this sense, we can say that the tautological map $\tQ_1 \to \tG_f$ gives a universal grading on the algebra $\tPi$.     

Now recall our fixed symmetrizer  $D = \diag(d_i \mid i \in I)$ and set $d \seq \gcd(d_i \mid i \in I)$.
Let $\Gamma' \subset \Gamma$ be the subgroup generated by $\{ \deg(a) \mid a \in \tQ_1\}$.
Note that $\Gamma'$ is a free abelian group with a basis $\{q^{2d}, t^{2}\} \cup \{q^{-d_if_{ij}}t\mu^{(g)}_{ij}\mid (i,j) \in \Omega, 1 \le g \le g_{ij} \}$.

\begin{Prop}
The degree map \eqref{eq:deg} gives an isomorphism $\deg \colon \tG_f \simeq \Gamma'$.
\end{Prop}
\begin{proof}
Choose integers $\{a_{i}\}_{i \in I}$ satisfying $\sum_{i \in I}a_i d_i = d$.  
Let $e$ and $w$ be the elements of $\tG_f$ given by $e \seq \prod_{i \in I}[\ep_i]^{a_i}$ and $w = [\alpha_{ij}^{(g)}] [\alpha_{ji}^{(g)}][\ep_{i}]^{f_{ij}}$
respectively. 
Note that $w$ does not depend on the choice of $i,j \in I$ with $i \sim j$ and $1 \le g \le g_{ij}$ by \eqref{eq:Whom}.
We define a group homomorphism $\iota \colon \Gamma' \to \tG_f$ by $\iota(q^{2d}) \seq e$, $\iota(t^2) \seq w$ and $\iota(q^{-d_if_{ij}}t\mu^{(g)}_{ij})\seq [\alpha^{(g)}_{ij}]$ for $(i,j) \in \Omega$, $1 \le g \le g_{ij}$.
It is easy to see $\deg \circ \iota = \id$.
Now we shall prove $\iota \circ \deg = \id$.
First, we observe that $[\ep_i]^{f_{ij}} = [\ep_j]^{f_{ji}}$ when $i \sim j$ by \eqref{eq:Whom}.
Since $f_{ij} = d_j / d_{ij}$, we have
\[ [\ep_i]^{r/d_i} = ([\ep_i]^{f_{ij}})^{rd_{ij}/d_id_j} = ([\ep_j]^{f_{ji}})^{rd_{ij}/d_id_j} = [\ep_j]^{r/d_j}\]
for any $i,j \in I$ with $i \sim j$.
Since $C$ is assumed to be irreducible, it follows that $[\ep_i]^{r/d_i} = [\ep_j]^{r/d_j}$ for any $i, j \in I$.
Furthermore, since $\tG_f$ is torsion-free, we get
\begin{equation} \label{eq:epd}
[\ep_i]^{d_j/d} = [\ep_j]^{d_i/d} \qquad \text{for any $i, j \in I$}.
\end{equation} 
Using \eqref{eq:epd}, for each $i \in I$, we find
\[ \iota( \deg [\ep_i]) = e^{d_i/d} = \prod_{j \in I}[\ep_j]^{a_jd_i/d} = \prod_{j \in I}[\ep_i]^{a_j d_j /d} = [\ep_i].\]
The equality $\iota(\deg [\alpha^{(g)}_{ij}]) = [\alpha^{(g)}_{ij}]$ is obvious. 
Thus we conclude that $\iota \circ \deg = \id$ holds. 
\end{proof}

In particular, we have the isomorphism of group rings $\Z[\tG_f] \simeq \Z[\Gamma']$.
Using the notation in the above proof, we consider the formal roots $e^{1/d}$ and $w^{1/2}$. Then we obtain the isomorphism
$\Z[\tG_f][e^{1/d}, w^{1/2}] \simeq \Z[\Gamma].$
This means that our deformed GCM $C(q,t,\mmu)$ can be specialized to any other deformation of $C$ which arises from a grading of the quiver $\tQ$ respecting the potential $W_\Omega$ (formally adding roots of deformation parameters if necessary).  

\subsection{$t$-Cartan matrices and representations of modulated graphs}
\label{Ssec:species}
In this subsection, we discuss the $t$-Cartan matrix $C(1,t)$, which is obtained from our $(q,t)$-deformed GCM $C(q,t)$ by evaluating the parameter $q$ at $1$. 
 Note that this kind of specialization is also studied by Kashiwara-Oh~\cite{KO} in the case of finite type very recently. Here we give an interpretation of the $t$-Cartan matrix from the viewpoint of certain graded algebras arising from an $F$-species. 

First, we briefly recall the notion of acyclic $F$-species over a base field $F$ \cite{Gab, Rin}. Let $I=\{1, \dots, n\}$.
By definition, an \emph{$F$-species} $(F_i, {}_iF_{j})$ over $F$ consists of 
\begin{itemize}
\item a finite dimensional skew-field $F_i$ over $F$ for each $i \in I$;
\item an $(F_i, F_j)$-bimodule ${}_i F_{j}$ for each $i, j \in I$ such that $F$ acts centrally on ${}_i F_{j}$ and $\dim_{F}{}_i F_{j}$ is finite;
\item There does not exist any sequence  $i_1, \dots, i_l, i_{l+1}=i_1$ such that ${}_{i_k} F_{i_{k+1}}\neq 0$ for each $k=1, \dots, l$. 
\end{itemize} 
For ${}_iF_j\neq 0$, we write ${}_{F_i}({}_iF_j) \simeq F_i^{\oplus {-c_{ij}}}$ and $({}_iF_j)_{F_j} \simeq F_j^{\oplus {-c_{ji}}}$. If we put $c_{ii}=2$ and $c_{ij}=0$ for ${}_iF_j=0={}_jF_i$, the matrix $C\seq ({c_{ij}})_{i,j \in I}$ is clearly a GCM with left symmetrizer $D = \diag(\dim _F F_i \mid i\in I )$. We have an acyclic orientation $\Omega$ of this GCM determined by the conditions ${}_iF_j\neq 0$. 
{Following} our convention in \S \ref{Ssec:notation}, we write $\dim_F F_i = {d_i}$. 
For our $F$-species $(F_i, {}_iF_j)$, we set $S \seq \prod_{i \in I} F_i$ and $B \seq \bigoplus_{(i,j) \in \Omega} {}_iF_j$. Note that $B$ is an $(S,S)$-bimodule. 
We define a finite dimensional hereditary algebra $T = {T(C, D, \Omega)}$ to be the tensor algebra $T \seq T_S(B)$. 
Note that we use the same convention for $T(C, D, \Omega)$ as that in Gei\ss-Leclerc-Schr\"oer~\cite{GLS3}, unlike our dual convention for the algebra $\Pi(\ell)$.

We can also define the preprojective algebra (see \cite{DR} for details). For $(i,j)\in \Omega$, there exists a $F_j$-basis $\{x_1, \dots, x_{{|c_{ji}|}}\}$ of ${}_iF_j$ and a $F_j$-basis $\{y_1, \dots, y_{{|c_{ji}|}}\}$ of $\Hom_{F_j}({}_iF_j, F_j)$ such that for every $x \in {}_iF_j$
we have
$x = \sum_{i=1}^{{|c_{ji}|}} y_i(x)x_i.$
We have the canonical element $\mathtt{c}_{ij} = \sum_{i=1}^{{|c_{ji}|}} x_i \otimes_{F_i} y_i \in {}_iF_j \otimes_{F_j} \Hom_{F_i}({}_iF_j, F_i)$ which does not depend on our choice of basis $\{x_i\}$ and $\{y_j\}$. Letting ${}_jF_i\coloneqq \Hom_{F_j}({}_iF_j, F_j)$ for $(i,j)\in \Omega$, we can also define the similar canonical element $\mathtt{c}_{ji}\in {}_jF_i \otimes_{F_i} {}_iF_j$.
We put $\overline{B}\coloneqq \bigoplus_{(i,j)\in \Omega}({}_iF_j \oplus {}_jF_i)$, and define the preprojective algebra $\Pi_T=\Pi_T(\ell)$ of the algebra $T$ as
$$T_S(\overline{B})/\langle \sum_{(i,j)\in \Omega}\mathrm{sgn}_{\Omega}(i,j)\mathtt{c}_{ij}\rangle.$$
Let $P^T_i$ (resp. $P^{\Pi_T}_i$) denote the indecomposable projective $T$-module (resp. $\Pi_T$-module) associated with $i$, and $\tau_T$ the Auslander-Reiten translation for (left) $T$-modules.
Note that this algebra $\Pi_T$ satisfies $P^{\Pi_T}_i=\bigoplus_{k\geq 0} \tau_T^{-k} P^T_i$ 
by an argument on the preprojective component of the Auslander-Reiten quiver of $T$ similar to \cite[Proposition 4.7]{So}. Note that our $F$-species $(F_i, {}_iF_j)$ is nothing but a \emph{modulated graph} associated with ${(C, D,\Omega)}$ in the sense of Dlab-Ringel~\cite{DR}, although we will work with these algebras along with a context of a deformation of $C$.

Although there is obviously no nontrivial $\Z$-grading on $S$ by the fact $F_i$ is a finite dimensional skew-field, we can nevertheless endow $T$ and $\Pi_T$ with a $t^{\Z}$-grading induced from their tensor algebra descriptions.
Each element of ${}_iF_j$ has degree $t$. We remark that if we specifically choose a decomposition of each ${}_iF_j$ like $F(\!(\ep)\!)$-species $\Tilde{H}$ in \cite[\S4.1]{GLS6} and define its preprojective algebra, then we can also endow these algebras with natural $\mmu^{\Z}$-gradings and homogeneous relations by using \cite[Lemma 1.1]{DR}. But we only consider the $t^{\Z}$-grading here since our aim is to interpret the $t$-Cartan matrix. 
By our $t^{\Z}$-grading, our algebra $\Pi_T$ satisfies the condition (A) in \S\ref{Ssec:pga} (with $\kk = F$).

We have the following complex of $t$-graded modules for each simple module $F_i$:

\begin{Lem}
\label{Thm:Pres2}
The complex
\begin{equation}
t^2 P^{\Pi_T}_i \xrightarrow{\psi^{(i)}}
\bigoplus_{j\sim i} (P_j^{\Pi_T})^{\oplus (-t{C_{ji}(1,t)})}
\rightarrow
P^{\Pi_T}_i \to F_i \to 0.
\end{equation}
is exact.
Moreover, the followings hold.
\begin{enumerate}
\item\label{res1sp} When $C$ is of infinite type, $\Ker \psi^{(i)} =0$ for all $i \in I$. In particular, each object in $\Pi_T\umod_{t^\Z}^+$ has projective dimension at most $2$.
\item\label{res2sp} When $C$ is of finite type, we have $\Ker \psi^{(i)} \cong t^h F_{i^*}$ for each $i \in I$.
\end{enumerate}
\end{Lem}
\begin{proof}
    The statement \eqref{res1sp} is deduced from the Auslander-Reiten theory for $T$ (e.g. \cite[Proposition 7.8]{AHIKM}). The statement \eqref{res2sp} follows from \cite[\S 6]{So}. Note that $C$ is of finite type if and only if $\Pi_T$ is a self-injective finite dimensional algebra and its Nakayama permutation can be similarly computed as Theorem~\ref{Thm:Pres} by an analogue of \cite[\S 3]{Miz} (see Remark~\ref{rem:AHIKM}).
\end{proof}
\begin{Cor} \label{Cor:chi2}
For any $i,j \in I$, the followings hold.
\begin{enumerate}
\item \label{Cor:chi2:1}
When $C$ is of finite type, we have 
\[ {
d_i \tC_{ij}(1,t) = 
\frac{t}{1-t^{2h}}\left(\dim_{t^{\Z}}(e_i P_j^{\Pi_T}) - t^h \dim_{t^{\Z}}(e_{i^*} P_j^{\Pi_T}) \right).}
\]
\item \label{Cor:chi2:2} When $C$ is of infinite type, we have
\[{
d_i\tC_{ij}(1,t) = t\dim_{t^\Z}(e_i P_j^{\Pi_T}).} \]
\end{enumerate}
Here $\dim_{t^\Z}$ denotes the graded dimension of $t^\Z$-graded $F$-vector spaces.
\end{Cor}
\begin{proof}
    The equality $[P_j^{\Pi_T}]=\sum_{i\in I} (\dim_{t^{\Z}}(e_iP_j^{\Pi_T}) / \dim_F F_i) [F_i]$ in $\hK(\Pi_T\umod_{t^\Z})$
and an equality $\dim_{t^{\Z}} e_i \Pi_T \otimes_{\Pi_T} F_j = \delta_{ij} {d_i}$
immediately yield our assertion by Lemma~\ref{Thm:Pres2} with arguments similar to the case of the generalized preprojective algebras in \S \ref{Ssec:EP}.
\end{proof}
\begin{Rem}\label{rem:AHIKM} In the case of our algebra $\Pi_T$, the two-sided ideal $J_i\coloneqq \Pi_T(1-e_i)\Pi_T$ and the ideal semi-group $\langle J_1, \dots, J_n \rangle$ also gives the Weyl group symmetry on its module category analogously to \cite{IR,BIRS,Miz} (see \cite[\S 7.1]{AHIKM}). Even if we consider the algebra $\Pi_T$ and $t^{\Z}$-homogeneous ideal $J_i$, we can also establish the similar braid group symmetry as \S\ref{sec:braid} after the specialization $q\rightarrow 1$ and $\mmu \rightarrow 1$ by Lemma~\ref{Thm:Pres2}.
    \end{Rem}
\begin{Rem}
The algebra $\Pi_T$ is a Koszul algebra for non-finite types and $(h-2, h)$-almost Koszul algebras for finite types in the sense of \cite{BBK} with our $t^{\Z}$-gradings. Thus Corollary~\ref{Cor:chi2} might be interpreted in the context of \cite[\S 3.3]{BBK}.
\end{Rem}

As a by-product of this description, we have the following generalization of the formula in \cite[Proposition 2.1]{HL15} and \cite[Proposition 3.8]{Fuj22} for any bipartite symmetrizable Kac-Moody type. 
For a $t$-series $f(t) = \sum_k f_k t^k \in \Z[\![t, t^{-1}]\!]$, we write $[f(t)]_k \seq f_k$ for $k \in \Z$.

\begin{Prop} \label{Prop:Ext}
Assume that $C$ is bipartite and take a height function $\xi$ for $C$ such that $\Omega_\xi = \Omega$ (see \S\ref{Ssec:cinv}).
Let $(F_i, {}_iF_j)$ be a modulated graph associated with ${(C, D,\Omega)}$ as above. 
Let $M \simeq \tau_T^{-k}P^T_i$ and $N \simeq \tau_T^{-l}P_j^T$ be any two indecomposable preprojective $T$-modules.
When $C$ is of infinite type, we have 
\begin{equation}
\label{eq:ExtT}
\dim_F \Ext_T^1(M,N) =
\left[{d_i\tC_{ij}(1,t)}\right]_{(\xi(i)+2k)-(\xi(j)+2l) -1}.\end{equation}
When $C$ is of finite type, the equality \eqref{eq:ExtT} still holds  provided that \begin{equation} \label{eq:1&h-1}
1 \le (\xi(i)+2k)-(\xi(j)+2l) -1 \le h-1.
\end{equation}
Otherwise, we have $\Ext_T^1(M,N) = 0$.
\end{Prop}

\begin{proof}
We may deduce the assertion by a combinatorial thought using the formula \eqref{eq:inv3} as in \cite{HL15} or \cite{Fuj22}. 
But, here we shall give another proof using the algebra $\Pi_T$. 

For any $t^\Z$-graded $T$-module $M$, we have a decomposition $M = \bigoplus_{u \in \Z} M^{[u]}$, where $M^{[u]} \seq \bigoplus_{i \in I}e_iM_{u - \xi(i)}$. 
Note that each $M^{[u]}$ is an $T$-submodule of $M$, 
since $\xi$ is a height function satisfying $\Omega_\xi = \Omega$. 
We have the following isomorphism
\begin{equation} \label{eq:isomHP}
{}_T(P_i^{\Pi_T})^{[u]} \cong \begin{cases}
\tau_T^{-k}P_i^{T} & \text{if $u = \xi(i) + 2k$ for $k \in \Z_{\ge 0}$}, \\
0 & \text{otherwise}
\end{cases} 
\end{equation}
as (ungraded) $T$-modules.
Now, we have for each $M \simeq \tau_T^{-k}P^T_i$ and $N \simeq \tau_T^{-l}P_j^T$
\begin{align*}
\dim_F \Ext_T^1(M,N) &=
\dim_F \Ext_T^1(\tau_T^{-k}P^T_i, \tau_T^{-l}P^T_j) \allowdisplaybreaks \\
&= \dim_F e_j\tau_T^{(k-l-1)}P^T_i&&\text{(cf.~\cite[\S IV 2.13]{ASS})} \allowdisplaybreaks \\
&= \dim_F e_j(P_i^{\Pi_T})^{[\xi(i)+2(k-l-1)]}&&\text{(\ref{eq:isomHP})} \allowdisplaybreaks \\
&= \dim_{F} (e_j{P}^{\Pi_T}_i)_{(\xi(i)+2k)-(\xi(j)+2l)-2}. 
\end{align*}
When $C$ is of infinite type, we deduce the desired equation \eqref{eq:ExtT} from Corollary~\ref{Cor:chi2}~\eqref{Cor:chi2:2}. 
When $C$ is of finite type, we can find that $(e_j {P}^{\Pi_T}_i)_{(\xi(i)+2k)-(\xi(j)+2l)-2}$ is non-zero only if the condition~\eqref{eq:1&h-1} is satisfied by an analogue of \cite[Corollary 3.9]{FM}.
When \eqref{eq:1&h-1} is satisfied, we get \eqref{eq:ExtT} by Corollary~\ref{Cor:chi2}~\eqref{Cor:chi2:1}.
\end{proof}
\begin{Rem}
    In \cite{GLS}, they also introduced the $1$-Iwanaga-Gorenstein algebra $H$ over any field $\Bbbk$ associated with a GCM $C$, its symmetrizer $D$, and an orientation $\Omega$.
    The algebra $H$ has quite similar features to our algebra $T$, and we can also show a version of Proposition~\ref{Prop:Ext} for the algebra $H$ with $t^{\mathbb{Z}}$-graded structure of the corresponding generalized preprojective algebra in a similar way.
    These algebras $H$ and $T$ have the following common dimension property of extension groups due to \cite[Proposition~5.5]{GLS3}:

    We keep the convention in Proposition~\ref{Prop:Ext}. Let $X \simeq \tau_H^{-k}P^H_i$ and $Y \simeq \tau_H^{-l}P_j^H$ be any two indecomposable preprojective $H$-modules. Then we have
    \[\dim_{\Bbbk} \Ext^1_H(X, Y) = \dim_F \Ext^1_T(M, N).\]
    Thanks to this common dimension property between the algebras $H$ and $T$, Corollary~\ref{Cor:chi} specializes to Corollary~\ref{Cor:chi2} after the specialization $q\rightarrow 1$ and $\mmu \rightarrow 1$ with Remark~\ref{Rem:GLS}.
\end{Rem}
\begin{Rem}
When the authors almost finished writing this paper, a preprint~\cite{KO23} by Kashiwara-Oh appeared in arXiv, which shows that the $t$-Cartan matrix of finite type is closely related to the representation theory of quiver Hecke algebra.
Combining their main theorem  with Proposition~\ref{Prop:Ext} above, we find a relationship between the representation theory of the modulated graphs and that of quiver Hecke algebras, explained as follows.

Let $C$ be a Cartan matrix of finite type, and let $\fg$ denote the simple Lie algebra associated with $C$. 
Let $R$ be the quiver Hecke algebra associated with $C$ and its minimal symmetrizer $D$, which categorifies the quantized enveloping algebra $U_q(\fg)$.
We are interested in the $\Z_{\ge 0}$-valued invariant $\fd(S, S')$ defined by using the R-matrices, which measures how far two ``affreal" $R$-modules $S$ and $S'$ are from being mutually commutative with respect to the convolution product (or parabolic induction).
Given an (acyclic) orientation $\Omega$ of $C$, we have an affreal $R$-module $S_\Omega(\alpha)$ for each positive root $\alpha$ of $\fg$, called a cuspidal module.   
See \cite{KO23} for details. 

On the other hand, we have a generalization of the Gabriel theorem for $F$-species (see \cite{DR2,DR3,Rin}).
In particular, for each positive root $\alpha$ of $\fg$, there exists an indecomposable module $M_\Omega(\alpha)$ over the algebra $T = {T(C,D,\Omega)}$ satisfying $\sum_{i \in I} (\dim_{F_i} e_i M_{\Omega}(\alpha)) \alpha_i = \alpha$, uniquely up to isomorphism. 
Note that every indecomposable $T$-module is a preprojective module when $C$ is of finite type.

Then, \cite[Main Theorem]{KO23} and Proposition~\ref{Prop:Ext} tell us that the equality 
\begin{equation} \label{eq:d=Ext} \fd\left(S_{{\Omega}^*}(\alpha), S_{{\Omega}^*}(\beta)\right) = \dim_F \Ext_T^1(M_\Omega(\alpha), M_\Omega(\beta)) + \dim_F \Ext_T^1(M_\Omega(\beta), M_\Omega(\alpha))
\end{equation}
holds for any positive roots $\alpha$ and $\beta$, where ${\Omega}^*$ denotes the orientation of $C$ opposite to $\Omega$.
In particular, \eqref{eq:d=Ext} implies that the following three conditions are mutually equivalent for any positive roots $\alpha$ and $\beta$ :
\begin{itemize}
\item The convolution product $S_{{\Omega}^*}(\alpha) \circ S_{{\Omega}^*}(\beta)$ is simple;
\item We have an isomorphism $S_{{\Omega}^*}(\alpha) \circ S_{{\Omega}^*}(\beta) \simeq S_{{\Omega}^*}(\beta) \circ S_{{\Omega}^*}(\alpha)$ of $R$-modules;  
\item We have $\Ext_T^1(M_\Omega(\alpha), M_\Omega(\beta)) = \Ext_T^1(M_\Omega(\beta), M_\Omega(\alpha)) =0.$ 
\end{itemize}
Note that an analogous statement in the case of fundamental modules over the quantum loop algebra of type $\mathrm{ADE}$ is obtained in \cite{Fuj22}.
\end{Rem}

\subsection*{Acknowledgments}
The authors are grateful to Christof Gei\ss, Naoki Genra, David Hernandez, Yuya Ikeda, Osamu Iyama, Bernhard Keller, Taro Kimura, Yoshiyuki Kimura, Bernard Leclerc and Hironori Oya for useful discussions and comments.
This series of works of the authors is partly motivated from the talk~\cite{Kels} given by Bernhard Keller. They thank him for sharing his ideas and answering some questions.
They were indebted to Laboratoire de Math\'ematiques \`a l'Universit\'e de Caen Normandie for the hospitality during their visit in the fall 2021. This work was partly supported by the Osaka City University Advanced Mathematical Institute: MEXT Joint Usage/Research Center on Mathematics and Theoretical Physics [JPMXP0619217849].

\bibliography{ref}
\end{document}